\documentclass[a4paper,11pt]{article}       
\usepackage{enumerate}          
\usepackage{amssymb,amsmath,amsfonts,amsxtra,amsthm}            
\usepackage{pdfsync}
\usepackage{xcolor}
\usepackage{url}
\usepackage{mathrsfs}
\usepackage{hyperref}
\usepackage{anysize}
\usepackage{comment}

\begin{document}
\theoremstyle{plain}
\newtheorem{definition}{Definition}
\newtheorem{theorem}{Theorem}
\newtheorem{proposition}[definition]{Proposition}
\newtheorem{lemma}[definition]{Lemma}
\newtheorem{corollary}[definition]{Corollary}
\newtheorem{example}[definition]{Example}
\newtheorem{conjecture}{Conjecture}
\newtheorem{problem}{Problem}
\newtheorem{assumption}{Assumption}
\newtheorem*{discussion}{Discussion}
\newtheorem{question}[equation]{Question}
\newtheorem{remark}[definition]{Remark}
\newtheorem{notation}{Notation}
\errorcontextlines=0

\newcommand{\C}{\mathbb{C}}
\newcommand{\R}{\mathbb{R}}
\newcommand{\N}{\mathbb{N}}
\renewcommand{\o}{\text{o}}
\newcommand{\ord}{\text{ord}}
\renewcommand{\d}{\partial}
\newcommand{\supp}{\text{Supp}}
\newcommand{\phg}{\text{phg}}
\renewcommand{\hom}{\text{hom}}
\newcommand{\sR}{\text{sR}}
\newcommand{\Op}{\text{Op}}
\newcommand{\e}{\varepsilon}
\renewcommand{\leq}{\leqslant}
\renewcommand{\geq}{\geqslant}
\renewcommand{\div}{\text{div}}
\newcommand{\todo}[1]{$\clubsuit$ {\tt #1}}

\title{Subelliptic wave equations are never observable}
\date{\today}
\author{Cyril Letrouit\footnote{Sorbonne Universit\'e, Universit\'e Paris-Diderot, CNRS, Inria, Laboratoire Jacques-Louis Lions,  F-75005 Paris (\texttt{letrouit@ljll.math.upmc.fr})}\ \footnote{DMA, \'Ecole normale sup\'erieure, CNRS, PSL Research University, 75005 Paris}}
\maketitle
\abstract{It is well-known that observability (and, by duality, controllability) of the elliptic wave equation, i.e., with a Riemannian Laplacian, in time $T_0$ is almost equivalent to the  Geometric Control Condition (GCC), which stipulates that any geodesic ray meets the control set within time $T_0$. We show that in the subelliptic setting, GCC is never verified, and that subelliptic wave equations are never observable in finite time. More precisely, given any subelliptic Laplacian $\Delta=-\sum_{i=1}^m X_i^*X_i$ on a manifold $M$, and any measurable subset $\omega\subset M$ such that $M\backslash \omega$ contains in its interior a point $q$ with $[X_i,X_j](q)\notin \text{Span}(X_1,\ldots,X_m)$ for some $1\leq i,j\leq m$, we show that for any $T_0>0$, the wave equation with subelliptic Laplacian $\Delta$ is not observable on $\omega$ in time $T_0$. 

\smallskip

The proof is based on the construction of sequences of solutions of the wave equation concentrating on geodesics (for the associated sub-Riemannian distance) spending a long time in $M\backslash \omega$. As a counterpart, we prove a positive result of observability for the wave equation in the Heisenberg group, where the observation set is a well-chosen part of the phase space.}

\setcounter{tocdepth}{1}
\tableofcontents

\section{Introduction}

\subsection{Setting} \label{s:setting}

Let $n\in\mathbb{N}^*$ and let $M$ be a smooth connected compact manifold of dimension $n$ with a non-empty boundary $\partial M$. Let $\mu$ be a smooth volume on $M$.
%
We consider $m\geq 1$ smooth vector fields $X_1,\ldots,X_m$ on $M$ which are not necessarily independent, and we assume that the following  H\"ormander condition holds (see \cite{hormander1967hypoelliptic}):

\begin{center}
The vector fields $X_1,\ldots,X_m$ and their iterated brackets $[X_i,X_j], [X_i,[X_j,X_k]]$, etc. span the tangent space $T_qM$ at every point $q\in M$.
 \end{center}

\smallskip
We consider the sub-Laplacian $\Delta$ defined by
\begin{equation*}
\Delta=-\sum_{i=1}^m X_i^*X_i=\sum_{i=1}^m X_i^2+\div_\mu(X_i)X_i
\end{equation*}
where the star designates the transpose in $L^2(M,\mu)$ and the divergence with respect to $\mu$ is defined by $L_X\mu=(\div_\mu X)\mu$, where $L_X$ stands for the Lie derivative. Then $\Delta$ is hypoelliptic (see \cite[Theorem 1.1]{hormander1967hypoelliptic}). 

\smallskip

We consider $\Delta$ with Dirichlet boundary conditions and the domain $D(\Delta)$ which is the completion in $L^2(M,\mu)$ of the set of all $u\in C_c^\infty(M)$ for the norm $\|(\text{Id}-\Delta) u\|_{L^2}$. We also consider the operator $(-\Delta)^{\frac12}$ with domain $D((-\Delta)^{\frac12})$ which is the completion in $L^2(M,\mu)$ of the set of all $u\in C_c^\infty(M)$ for the norm $\|(\text{Id}-\Delta)^{\frac12} u\|_{L^2}$.

\smallskip

Consider the wave equation
\begin{equation} \label{e:system}
\left\lbrace \begin{array}{l}
\partial_{tt}^2u-\Delta u=0\, \text{ \ in $(0,T)\times M$}\\
u=0 \text{ \ on } (0,T)\times \partial M, \\
(u_{|t=0},\partial_tu_{|t=0})=(u_0,u_1)\,
\end{array}\right.
\end{equation}
where $T>0$. It is well-known (see for example \cite[Theorem 2.1]{garetto2015wave}, \cite[Chapter II, Section 6]{engel1999one}) that for any $(u_0,u_1)\in D((-\Delta)^{\frac12})\times L^2(M)$, there exists a unique solution
\begin{equation} \label{e:spacesolutions}
u\in C^0(0,T;D((-\Delta)^{\frac12}))\;\cap\; C^1(0,T;L^2(M))
\end{equation}
to \eqref{e:system} (in a mild sense). 

\smallskip

We set
\begin{equation}  \label{e:defmathcalH}
\|v\|_{\mathcal{H}}=\left(\int_M \sum_{j=1}^m (X_jv(x))^2d\mu(x)\right)^{\frac12}.
\end{equation}
 Note that $\|v\|_{\mathcal{H}}=\|(-\Delta)^{\frac12}v\|_{L^2(M,\mu)}$.

\smallskip

The natural energy of a solution is
\begin{equation*}
E(u(t,\cdot))=\frac12(\|\partial_tu(t,\cdot)\|_{L^2(M,\mu)}^2+\|u(t,\cdot)\|^2_{\mathcal{H}}).
\end{equation*}
If $u$ is a solution of \eqref{e:system}, then 
\begin{equation*}
\frac{d}{dt}E(u(t,\cdot))=0,
\end{equation*}
and therefore the energy of $u$ at any time is equal to
\begin{equation*}
\|(u_0,u_1)\|^2_{\mathcal{H}\times L^2}=\|u_0\|^2_{\mathcal{H}}+\|u_1\|^2_{L^2(M,\mu)}.
\end{equation*}

In this paper, we investigate exact observability for the wave equation \eqref{e:system}. 

\begin{definition} \label{d:obs}
Let $T_0>0$ and $\omega \subset M$ be a $\mu$-measurable subset. The subelliptic wave equation \eqref{e:system} is exactly observable on $\omega$ in time $T_0$ if there exists a constant $C_{T_0}(\omega)>0$ such that, for any $(u_0,u_1)\in D((-\Delta)^{\frac12})\times L^2(M)$, the solution $u$ of \eqref{e:system} satisfies
\begin{equation} \label{e:strongobs}
\int_0^{T_0} \int_\omega|\partial_tu(t,x)|^2d\mu(x)dt \geq C_{T_0}(\omega) \|(u_0,u_1)\|^2_{\mathcal{H}\times L^2}.
\end{equation}
\end{definition}

\subsection{Main result} \label{s:mainresult}
Our main result is the following.
\begin{theorem}  \label{t:main}
Let $T_0>0$ and let $\omega\subset M$ be a measurable subset. We assume that there exist $1\leq i,j\leq m$ and $q$ in the interior of $M\backslash\omega$ such that $[X_i,X_j](q)\notin {\rm Span}(X_1(q),\ldots,X_m(q))$. Then the subelliptic wave equation \eqref{e:system} is not exactly observable on $\omega$ in time $T_0$.
\end{theorem}

Consequently, using a duality argument (see Section \ref{s:contr}), we obtain that exact controllability does not hold either in any finite time.

\begin{definition} \label{d:cont}
Let $T_0>0$ and $\omega \subset M$ be a measurable subset. The subelliptic wave equation \eqref{e:system} is exactly controllable on $\omega$ in time $T_0$ if for any $(u_0,u_1)\in D((-\Delta)^{\frac12})\times L^2(M)$, there exists $g\in L^2((0,T_0)\times M)$ such that the solution $u$ of 
\begin{equation} \label{e:systemcont}
\left\lbrace \begin{array}{l}
\partial_{tt}^2u-\Delta u=\bold{1}_\omega g\, \text{ \ in $(0,T_0)\times M$}\\
u=0 \text{ \ on } (0,T_0)\times \partial M, \\
(u_{|t=0},\partial_tu_{|t=0})=(u_0,u_1)\,
\end{array}\right.
\end{equation}
satisfies $u(T_0,\cdot)=0$.
\end{definition}

\begin{corollary} \label{c:contr}
Let $T_0>0$ and let $\omega\subset M$ be a measurable subset. We assume that there exist $1\leq i,j\leq m$ and $q$ in the interior of $M\backslash\omega$ such that $[X_i,X_j](q)\notin {\rm Span}(X_1(q),\ldots,X_m(q))$.  Then the subelliptic wave equation \eqref{e:system} is not exactly controllable on $\omega$ in time $T_0$.
\end{corollary}

In what follows, we denote by $\mathcal{D}$ the set of all vector fields that can be decomposed as linear combinations with smooth coefficients of the $X_i$:
$$
\mathcal{D}={\rm Span}(X_1,\ldots,X_m)\subset TM.
$$
$\mathcal{D}$ is called the \emph{distribution} associated to the vector fields $X_1,\ldots,X_m$. For $q\in M$, we denote by $\mathcal{D}_q\subset T_qM$ the distribution $\mathcal{D}$ taken at point $q$.


%

\smallskip

The assumptions of Theorem \ref{t:main} are satisfied as soon as the interior $U$ of $M\setminus \omega$ is non-empty and $\mathcal{D}$ has constant rank $<n$ in $U$. Indeed, under these conditions, we can argue by contradiction: assume that for any $q\in U$ and any $1\leq i,j\leq m$, there holds $[X_i,X_j](q)\in {\rm Span}(X_1(q),\ldots,X_m(q))=\mathcal{D}_q$. Then we have $[\mathcal{D},\mathcal{D}]\subset \mathcal{D}$ in $U$, i.e., $\mathcal{D}$ is involutive. By Frobenius's theorem, $\mathcal{D}$ is then completely integrable, which contradicts H\"ormander's condition.

\smallskip

The following examples show that the assumptions of Theorem \ref{t:main} are also satisfied in some non-constant rank cases:

\begin{example}\label{ex:exempleBaouendi}
In the Baouendi-Grushin case, for which $X_1=\partial_{x_1}$ and $X_2=x_1\partial_{x_2}$ are vector fields on $(-1,1)_{x_1}\times\mathbb{T}_{x_2}$ where $\mathbb{T}=\R/\mathbb{Z}$, the corresponding  sub-Laplacian $\Delta=X_1^2+X_2^2$ (here, $\mu=dx_1dx_2$ for simplicity) is elliptic outside of the singular submanifold $S=\{x_1=0\}$. Therefore, the corresponding subelliptic wave equation is observable on any open subset containing $S$ (with some finite minimal time of observability, see \cite{bardos1992sharp}), but according to Theorem \ref{t:main}, it is not observable in any finite time on any subset $\omega$ such that the interior of $M\setminus \omega$ has a non-empty intersection with $S$. 
\end{example}

\begin{example}
In the Martinet case, the vector fields are $X_1=\partial_{x_1}$ and $X_2=\partial_{x_2}+x_1^2\partial_{x_3}$ on $(-1,1)_{x_1}\times\mathbb{T}_{x_2}\times\mathbb{T}_{x_3}$, and the corresponding sub-Laplacian is $\Delta=X_1^2+X_2^2$ (again, $\mu=dx_1dx_2dx_3$ for simplicity). Then, we have $[X_1,X_2]=2x_1\partial_{x_3}$. The only points at which this bracket belongs to the distribution ${\rm Span}(X_1,X_2)$ are the points for which $x_1=0$. Since this set of points has empty interior, the assumptions of Theorem \ref{t:main} are satisfied as soon as $M\setminus\omega$ has non-empty interior.
\end{example}

\begin{remark} 
The assumption of compactness on $M$ is not necessary: we may remove it, and just require that the subelliptic wave equation \eqref{e:system} in $M$ is well-posed. It is for example the case if $M$ is complete for the sub-Riemannian distance induced by $X_1,\ldots,X_m$ since $\Delta$ is then essentially self-adjoint (\cite{strichartz1986sub}).
\end{remark}

\begin{remark}
Theorem \ref{t:main} remains true if $M$ has no boundary. In this case, the equation \eqref{e:system} is well-posed in a space slightly smaller than \eqref{e:spacesolutions}: a condition of null average has to be added since non-zero constant functions on $M$ are solutions of \eqref{e:system}, see Section \ref{s:truncatedobsstatement}. The observability inequality of Theorem \ref{t:main} remains true in this space of solutions: anticipating the proof, we notice that the spiraling normal geodesics of Proposition \ref{p:existencegeod} still exist (since their construction is purely local), and we subtract to the initial datum $u_0^k$ of the localized solutions constructed in Proposition \ref{p:exactgb} their spatial average $\int_M u_0^k d\mu$.
\end{remark}

\begin{remark} \label{r:halfwave}
Thanks to abstract results (see for example \cite{miller2012}), Theorem \ref{t:main} remains true when the subelliptic wave equation \eqref{e:system} is replaced by the subelliptic half-wave equation $\partial_tu+i\sqrt{-\Delta}u=0$ with Dirichlet boudary conditions. 
\end{remark}

\subsection{Ideas of the proof} \label{s:ideas}
In the sequel, we call ``normal geodesic''\footnote{This terminology is common in sub-Riemannian geometry, and it is justified by the fact that we can naturally associate to the vector fields $X_1,\ldots,X_m$ a metric structure on $M$ for which these projected paths are geodesics, see \cite{montgomery2002tour}.} the projection on $M$ of a bicharacteristic (parametrized by time) for the principal symbol of the wave equation \eqref{e:system}. We will give a more detailed definition in Section \ref{s:normal}. 

\smallskip

The proof of Theorem \ref{t:main} mainly requires two ingredients:
\begin{enumerate}
\item There exist solutions of the free subelliptic wave equation \eqref{e:system} whose energy concentrates along any given normal geodesic;
\item There exist normal geodesics which ``spiral" around curves transverse to $\mathcal{D}$, and which therefore remain arbitrarily close to their starting point on arbitrarily large time-intervals.
\end{enumerate}
Combining the two above facts, the proof of Theorem \ref{t:main} is straightforward (see Section \ref{s:concl}). Note that the first point follows from the general theory of propagation of complex Lagrangian spaces, while the second point is the main novelty of this paper.

\smallskip

Since our construction is purely local (meaning that it does not ``feel" the boundary and only relies on the local structure of the vector fields), we can focus on the case where there is a (small) open neighborhood $V$ of the origin $O$ such that $V\subset M\backslash\omega$, and $[X_i,X_j](O)\notin \mathcal{D}_O$ for some $1\leq i,j\leq m$. In the sequel, we assume it is the case. 

\smallskip

Let us give an example of vector fields where the spiraling normal geodesics used in the proof of Theorem \ref{t:main} are particularly simple. We consider the three-dimensional manifold with boundary $M_{1}=(-1,1)_{x_1}\times \mathbb{T}_{x_2}\times \mathbb{T}_{x_3}$, where $\mathbb{T}=\R/\mathbb{Z}\approx (-1,1)$ is the 1D torus. We endow $M_{1}$ with the vector fields $X_1=\partial_{x_1}$ and $X_2=\partial_{x_2}-x_1\partial_{x_3}$. This is the ``Heisenberg manifold with boundary". We endow $M_1$ with an arbitrary smooth volume $\mu$. The normal geodesics we consider are given by
\begin{equation} \label{e:geodspiraling}
\begin{array}{l}
x_1(t)=\varepsilon\sin(t/\varepsilon)\\
x_2(t)=\varepsilon\cos(t/\varepsilon)-\e \\
x_3(t)=\varepsilon(t/2-\varepsilon\sin(2t/\varepsilon)/4).
\end{array}
\end{equation}
They spiral around the $x_3$ axis $x_1=x_2=0$.

\smallskip

Here, one should think of $\varepsilon$ as a small parameter. In the sequel, we denote by $x_\varepsilon$ the normal geodesic with parameter $\varepsilon$.

\smallskip

Clearly, given any $T_0>0$, for $\varepsilon$ sufficiently small, we have $x_\varepsilon(t)\in V$ for every $t\in (0,T_0)$. Our objective is to construct solutions $u^k$ of the subelliptic wave equation \eqref{e:system} such that  $\|(u^k_0,u^k_1)\|_{\mathcal{H}\times L^2}=1$ and the energy of $u^k(t,\cdot)$ concentrates outside of an open set $V_t$ containing $x_\varepsilon(t)$, i.e.,
\begin{equation*} 
 \int_{M_{1}\backslash V_t}\left(|\partial_tu^k(t,x)|^2+(X_1u_k(t,x))^2+(X_2u_k(t,x))^2\right)d\mu(x)
\end{equation*}
tends to $0$ as $k\rightarrow +\infty$ uniformly with respect to $t\in (0,T_0)$. As a consequence, the observability inequality \eqref{e:strongobs} fails.

\smallskip

The construction of solutions of the free wave equation whose energy concentrates on geodesics is classical in the elliptic (or Riemannian) case: these are the so-called Gaussian beams, for which a construction can be found for example in \cite{ralston1982gaussian}. Here, we adapt this construction to our subelliptic (sub-Riemannian) setting, which does not raise any problem since the normal geodesics we consider stay in the elliptic part of the operator $\Delta$. It may also be directly justified with the theory of propagation of complex Lagrangian spaces (see Section \ref{s:constrgbsmain}).

\smallskip

In the case of general vector fields $X_1,\ldots,X_m$, the existence of spiraling normal geodesics also has to be justified. For that purpose, we first approximate $X_1,\ldots,X_m$ by their \emph{nilpotent approximations}, and we then prove that for the latters, such a family of spiraling normal geodesics exists, as in the Heisenberg case. 

\subsection{Normal geodesics} \label{s:normal}

In this section, we explain in more details what \emph{normal geodesics} are. As said before, they are natural extensions of Riemannian geodesics since they are projections of bicharacteristics. 

\smallskip

We denote by $S_{\phg}^m(T^*((0,T)\times M))$ the set of polyhomogeneous symbols of order $m$ with compact support and by $\Psi_{\phg}^m((0,T)\times M)$ the set of associated polyhomogeneous pseudodifferential operators of order $m$ whose distribution kernel has compact support in $(0,T)\times M$ (see Appendix \ref{a:pseudo}). 

\smallskip

We set $P=\partial_{tt}^2-\Delta\in \Psi_\phg^2((0,T)\times M)$, whose principal symbol is
\begin{equation*}
p_2(t,\tau,x,\xi)=-\tau^2+g^*(x,\xi)
\end{equation*}
with $\tau$ the dual variable of $t$ and $g^*$ the principal symbol of $-\Delta$. For $\xi\in T^*M$, we have (see Appendix \ref{a:pseudo})
\begin{equation*}
g^*=\sum_{i=1}^m h_{X_i}^2.
\end{equation*}
Here, given any smooth vector field $X$ on $M$, we denoted by $h_X$ the Hamiltonian function (momentum map) on $T^*M$ associated with $X$ defined in local $(x,\xi)$-coordinates by $h_X(x,\xi)=\xi(X(x))$. 

\smallskip

In $T^*(\R\times M)$, the Hamiltonian vector field $\vec{p}_2$ associated with $p_2$ is given by $\vec{p}_2f=\{p_2,f\}$ where $\{\cdot,\cdot\}$ denotes the Poisson bracket (see Appendix \ref{a:pseudo}). Since $\vec{p}_2p_2=0$, we get that $p_2$ is constant along the integral curves of $\vec{p}_2$. Thus, the characteristic set $\mathcal{C}(p_2)=\{p_2=0\}$ is preserved by the flow of $\vec{p}_2$. Null-bicharacteristics are then defined as the maximal integral curves of $\vec{p}_2$ which live in $\mathcal{C}(p_2)$.  In other words, the null-bicharacteristics are the maximal solutions of
\begin{equation} \label{e:bicarac1}
\left\lbrace \begin{array}{l}
\dot{t}(s)=-2\tau(s)\,,\\
\dot{x}(s)=\nabla_\xi g^*(x(s),\xi(s))\,,\\
\dot{\tau}(s)=0\,,\\
\dot{\xi}(s)=-\nabla_xg^*(x(s),\xi(s))\,,\\
\tau^2(0)=g^*(x(0),\xi(0)).
\end{array}\right.
\end{equation}
This definition needs to be adapted when the null-bicharacteristic meets the boundary $\partial M$, but in the sequel, we only consider solutions of \eqref{e:bicarac1} on time intervals where $x(t)$ does not reach $\partial M$.

\smallskip

In the sequel, we take $\tau=-1/2$, which gives $g^*(x(s),\xi(s))=1/4$. This also implies that $t(s)=s+t_0$ and, taking $t$ as a time parameter, we are led to solve
\begin{equation} \label{e:bicarac}
\left\lbrace \begin{array}{l}
\dot{x}(t)=\nabla_\xi g^*(x(t),\xi(t))\,,\\
\dot{\xi}(t)=-\nabla_xg^*(x(t),\xi(t))\,,\\
g^*(x(0),\xi(0))=\frac{1}{4}.
\end{array}\right.
\end{equation}
In other words, the $t$-variable parametrizes null-bicharacteristics in a way that they are traveled at speed $1$.

\begin{remark}  \label{r:tangentbundle}
In the subelliptic setting, the co-sphere bundle $S^*M$ can be decomposed as $S^*M=U^*M\cup S\Sigma$, where $U^*M=\{g^*=1/4\}$ is a cylinder bundle, $\Sigma=\{g^*=0\}$ is the characteristic cone and $S\Sigma$ is the sphere bundle of $\Sigma$ (see \cite[Section 1]{de2018spectral}).
\end{remark}

We denote by $\phi_t:S^*M\rightarrow S^*M$ the (normal) geodesic flow defined by $\phi_t(x_0,\xi_0)=(x(t),\xi(t))$, where $(x(t),\xi(t))$ is a solution of the system given by the first two lines of \eqref{e:bicarac} and initial conditions $(x_0,\xi_0)$. Note that any point in $S\Sigma$ is a fixed point of $\phi_t$, and that the other normal geodesics are traveled at speed $1$ since we took $g^*=1/4$ in $U^*M$ (see Remark \ref{r:tangentbundle}).

\smallskip

The curves $x(t)$ which solve \eqref{e:bicarac} are geodesics (i.e. local minimizers) for a sub-Riemannian metric $g$ (see \cite[Theorem 1.14]{montgomery2002tour}).

\subsection{Observability in some regions of phase-space} \label{s:truncatedobsstatement}
We have explained in Section \ref{s:ideas} that the existence of solutions of the subelliptic wave equation \eqref{e:system} concentrated on spiraling normal geodesics is an obstruction to observability in Theorem \ref{t:main}. Our goal in this section is to state a result ensuring observability if one ``removes" in some sense these normal geodesics. 

\smallskip

For this result, we focus on a version of the Heisenberg manifold described in Section \ref{s:ideas} which has \emph{no boundary}. This technical assumption avoids us using boundary microlocal defect measures in the proof, which, in this sub-Riemannian setting, are difficult to handle. As a counterpart, we need to consider solutions of the wave equation with null initial average, in order to get well-posedness.

\smallskip

We consider the Heisenberg group $G$, that is $\R^3$ with the composition law
\begin{equation*}
(x_1,x_2,x_3)\star (x_1',x_2',x_3')=(x_1+x_1',x_2+x_2',x_3+x_3'-x_1x_2').
\end{equation*}
Then $X_1=\partial_{x_1}$ and $X_2=\partial_{x_2}-x_1\partial_{x_3}$ are left invariant vector fields on $G$. Since $\Gamma=\sqrt{2\pi}\mathbb{Z}\times \sqrt{2\pi}\mathbb{Z}\times 2\pi\mathbb{Z}$ is a co-compact subgroup of $G$, the left quotient $M_H=\Gamma \backslash G$ is a compact three dimensional manifold and, moreover, $X_1$ and $X_2$ are well-defined as vector fields on the quotient. We call $M_H$ endowed with the vector fields $X_1$ and $X_2$ the ``Heisenberg manifold without boundary". Finally, we define the Heisenberg Laplacian $\Delta_H=X_1^2+X_2^2$ on $M_H$. Since $[X_1,X_2]=-\partial_{x_3}$, it is a hypoelliptic operator.  We endow $M_H$ with an arbitrary smooth volume $\mu$. 

\smallskip

We introduce the space
\begin{equation*}
L^2_0=\left\{u_0\in L^2(M_H), \ \int_{M_H}u_0\;d\mu=0\right\}
\end{equation*}
and we consider the operator $\Delta_H$ whose domain $D(\Delta_H)$ is the completion in $L^2_0$ of the set of all $u\in C_c^\infty(M_H)$ with null-average for the norm $\|(\text{Id}-\Delta_H) u\|_{L^2}$.  Then, $-\Delta_H$ is definite positive and we consider $(-\Delta_H)^{\frac12}$ with domain $D((-\Delta_H)^{\frac12})=\mathcal{H}_0:=L^2_0 \; \cap \; \mathcal{H}(M_H)$. The wave equation
\begin{equation}  \label{e:wavenullaverage}
\left\lbrace \begin{array}{l}
\partial_{tt}^2u-\Delta_H u=0\, \text{ \ in $\R\times M_H$}\\
(u_{|t=0},\partial_tu_{|t=0})=(u_0,u_1)\in D((-\Delta_H)^{\frac12})\times L^2_0\,
\end{array}\right.
\end{equation}
admits a unique solution $u\in C^0(\R;D((-\Delta_H)^{\frac12}))\;\cap\; C^1(\R;L^2_0)$.

\smallskip

We note that $-\Delta_H$ is invertible in $L^2_0$. The space $\mathcal{H}_0$ is endowed with the norm $\|u\|_{\mathcal{H}}$ (defined in \eqref{e:defmathcalH} and also equal to $\|(-\Delta_H)^{\frac12}u\|_{L^2}$), and its topological dual $\mathcal{H}_0'$ is endowed with the norm $\|u\|_{\mathcal{H}_0'}:=\|(-\Delta_H)^{-\frac{1}{2}}u\|_{L^2}$. 

\smallskip

We note that $g^*(x,\xi)=\xi_1^2+(\xi_2-x_1\xi_3)^2$ and hence the null-bicharacteristics are solutions of
\begin{equation}\label{e:bicharH}
\begin{aligned}
&\dot{x}_1(t)=2\xi_1, &\quad & \dot{\xi}_1(t)=2\xi_3(\xi_2-x_1\xi_3),  \\
 &\dot{x}_2(t)=2(\xi_2-x_1\xi_3), &\quad& \dot{\xi}_2(t)=0,  \\
 &\dot{x}_3(t)=-2x_1(\xi_2-x_1\xi_3), &\quad& \dot{\xi}_3(t)=0.
\end{aligned}
\end{equation}
The spiraling normal geodesics described in Section \ref{s:ideas} correspond to $\xi_{1}=\cos(t/\e)/2$, $\xi_{2}=0$ and $\xi_{3}=1/(2\e)$. In particular, the constant $\xi_{3}$ is a kind of rounding number reflecting the fact that the normal geodesic spirals at a certain speed around the $x_3$ axis. Moreover, $\xi_{3}$ is preserved under the flow (somehow, the Heisenberg flow is completely integrable), and this property plays a key role in the proof of Theorem \ref{t:truncatedobs} below and justifies that we state it only for the Heisenberg manifold (without boundary).

\smallskip

As said above, normal geodesics corresponding to a large momentum $\xi_{3}$ are precisely the ones used to contradict observability in Theorem \ref{t:main}. We  expect to be able to establish observability if we consider only solutions of \eqref{e:system} whose $\xi_{3}$ (in a certain sense) is not too large. This is the purpose of our second main result.

\smallskip

Set 
\begin{equation*} 
V_\varepsilon=\left\{(x,\xi)\in T^*M_H : |\xi_{3}|> \frac{1}{\e}(g_x^*(\xi))^{1/2}\right\}
\end{equation*} 
Note that since $\xi_{3}$ is constant along null-bicharacteristics, $V_\varepsilon$ and its complementary $V_\varepsilon^c$ are invariant under the bicharacteristic equations \eqref{e:bicharH}.

\smallskip

In the next statement, we call horizontal strip the periodization under the action of the co-compact subgroup $\Gamma$ of a set of the form
\begin{equation*}
\{(x_1,x_2,x_3)\;: \; (x_1,x_2)\in[0,\sqrt{2\pi})^2, \; x_3\in I\}
\end{equation*}
where $I$ is a strict open subinterval of $[0,2\pi)$.

\begin{theorem} \label{t:truncatedobs}
Let $B\subset M_H$ be an open subset and suppose that $B$ is sufficiently small, so that $\omega=M_H\backslash B$ contains a horizontal strip. Let $a\in S_{\phg}^0(T^*M_H)$, $a\geq 0$, such that, denoting by $j:T^*\omega\rightarrow T^*M_H$ the canonical injection,
\begin{equation*}
j(T^*\omega)\cup V_\e\subset \text{Supp}(a)\subset T^*M_H,
\end{equation*} 
and in particular $a$ does not depend on time. There exists $\kappa>0$ such that for any $\varepsilon>0$ and any $T\geq \kappa \varepsilon^{-1}$, there holds
\begin{equation} \label{e:obstruncated}
C\|(u(0),\partial_tu(0))\|^2_{\mathcal{H}_0\times L^2_0}\leq \int_0^T|(\Op(a)\partial_tu,\partial_tu)_{L^2}|\;dt\;+\;\|(u(0),\partial_tu(0))\|^2_{L^2_0\times \mathcal{H}_0'}
\end{equation}
for some $C=C(\varepsilon,T)>0$ and for any solution $u\in C^0(\R;D((-\Delta_H)^{\frac12}))\;\cap\; C^1(\R;L^2_0)$ of \eqref{e:wavenullaverage}.
\end{theorem}
The term $\|(u_0,u_1)\|^2_{L^2\times \mathcal{H}_0'}$ in the right-hand side of \eqref{e:obstruncated} cannot be removed, i.e. our statement only consists in a \emph{weak} observability inequality. Indeed, the usual way to remove such terms is to use a unique continuation argument for eigenfunctions $\varphi$ of $\Delta$, but here it does not work since 
$\Op(a)\varphi=0$ does not imply in general that $\varphi\equiv0$ in the whole manifold, even if the support of $a$ contains $j(T^*\omega)$ for some non-empty open set $\omega$: in some sense, there is no ``pseudodifferential unique continuation argument''.

\subsection{Comments on the existing literature} \label{s:links}

\paragraph{Elliptic and subelliptic waves.} The exact controllability/observability of the elliptic wave equation is known to be almost equivalent to the so-called Geometric Control Condition (GCC) (see \cite{bardos1992sharp}) that any geodesic enters the control set $\omega$ within time $T$. In some sense, our main result is that GCC is not verified in the subelliptic setting, as soon as $M\backslash \omega$ contains in its interior a point $x$ at which $\Delta$ is ``truly subelliptic''. For the elliptic wave equation, in many geometrical situations, there exists a minimal time $T_0>0$ such that observability holds only for $T\geq T_0$: when there exists a geodesic $\gamma:(0,T_0)\rightarrow M$ traveled at speed $1$ which does not meet $\overline{\omega}$, one constructs a sequence of initial data $(u^k_0,u^k_1)_{k\in\N^*}$ of the wave equation whose associated microlocal defect measure is concentrated on $(x_0,\xi_0)\in S^*M$ taken to be the initial conditions for the null-bicharacteristic  projecting onto $\gamma$. Then, the associated sequence of solutions $(u^k)_{k\in\N^*}$ of the wave equation has an associated microlocal defect measure $\nu$ which is invariant under the geodesic flow: $\vec{p} \nu=0$ where $\vec{p}$ is the Hamiltonian flow associated to the principal symbol $p$ of the wave operator. In particular, denoting by $\pi:T^*M\rightarrow M$ the canonical projection, $\pi_*\nu$ gives no mass to $\omega$ since $\gamma$ is contained in $M\setminus \overline{\omega}$, and this proves that observability cannot hold.

\smallskip

In the subelliptic setting, the invariance property $\vec{p}\nu=0$ does not give any information on $\nu$ on the characteristic manifold $\Sigma$, since $\vec{p}=-2\tau\partial_t+\vec{g}^*$ vanishes on $\Sigma$. This is related to the lack of information on propagation of singularities in this characteristic manifold, see the main theorem of \cite{lascar1982}. If one instead tries to use the propagation of the microlocal defect measure for subelliptic half-wave equations, one is immediately confronted with the fact that $\sqrt{-\Delta}$ is not a pseudodifferential operator near $\Sigma$. 

\smallskip

This is why, in this paper, we used only the elliptic part of the symbol $g^*$ (or, equivalently, the strictly hyperbolic part of $p_2$), where the propagation properties can be established, and then the problem is reduced to proving geometric results on normal geodesics.

\paragraph{Subelliptic Schr\"odinger equations.} The recent article \cite{burq2019time} deals with the same observability problem, but for subelliptic Schr\"odinger equations: namely, the authors consider the (Baouendi)-Grushin Schr\"odinger equation $i\partial_tu-\Delta_G u=0$, where $u\in L^2((0,T)\times M_G)$, $M_G=(-1,1)_x\times \mathbb{T}_y$ and $\Delta_G=\partial_{x}^2+x^2\partial_{y}^2$ is the  Baouendi-Grushin Laplacian. Given a control set of the form $\omega=(-1,1)_x\times \omega_y$, where $\omega_y$ is an open subset of $\mathbb{T}$, the authors prove the existence of a minimal time of control $\mathcal{L}(\omega)$ related to the maximal height of a horizontal strip contained in $M_G\backslash \omega$. The intuition is that there are solutions of the  Baouendi-Grushin Schr\"odinger equation which travel along the degenerate line $x=0$ at a finite speed: in some sense, along this line, the Schr\"odinger equation behaves like a classical (half)-wave equation. What we want here is to explain in a few words why there is a minimal time of observability for the Schr\"odinger equation, while the wave equation is never observable in finite time as shown by Theorem \ref{t:main}.

\smallskip

The plane $\R^2_{x,y}$ endowed with the vector fields $\partial_x$ and $x\partial_y$ also admits normal geodesics similar to the $1$-parameter family $q_\varepsilon$, namely, for $\e>0$, 
\begin{align}
x(t)&=\varepsilon\sin(t/\varepsilon) \nonumber\\
y(t)&=\varepsilon(t/2-\varepsilon\sin(2t/\varepsilon)/4)\nonumber 
\end{align}
These normal geodesics, denoted by $\gamma_\varepsilon$, also ``spiral" around the line $x=0$ more and more quickly as $\varepsilon\rightarrow 0$, and so we might expect to construct solutions of the  Baouendi-Grushin Schr\"odinger equation with energy concentrated along $\gamma_\varepsilon$, which would contradict observability when $\varepsilon\rightarrow 0$ as above for the Heisenberg wave equation. 

\smallskip

However, we can convince ourselves that it is not possible to construct such solutions: in some sense, the dispersion phenomena of the Schr\"odinger equation exactly compensate the lengthening of the normal geodesics $\gamma_\varepsilon$ as $\varepsilon\rightarrow 0$ and explain that even these Gaussian beams may be observed in $\omega$ from a certain minimal time $\mathcal{L}(\omega)>0$ which is uniform in $\varepsilon$.

\smallskip

To put this argument into a more formal form, we consider the solutions of the bicharacteristic equations for the  Baouendi-Grushin Schr\"odinger equation $i\partial_tu-\Delta_Gu=0$ given by
\begin{align}
x(t)&=\varepsilon\sin(\xi_yt) \nonumber\\
y(t)&=\varepsilon^2\xi_y\left(\frac{t}{2}-\frac{\sin(2\xi_yt)}{4\xi_y}\right)\nonumber \\
\xi_x(t)&=\varepsilon\xi_y\cos(\xi_yt)\nonumber \\
\xi_y(t)&=\xi_y. \nonumber
\end{align}
It follows from the hypoellipticity of $\Delta_G$ (see \cite[Section 3]{burq2019time} for a proof) that 
\begin{equation*}
|\xi_y|^{1/2}\lesssim \sqrt{-\Delta_G}=(|\xi_x|^2+x^2|\xi_y|^2)^{1/2}=\varepsilon|\xi_y|.
\end{equation*}
Therefore $\varepsilon^2|\xi_y|\gtrsim 1$, and hence $|y(t)|\gtrsim t$, independently from $\varepsilon$ and $\xi_y$. This heuristic gives the intuition that a minimal time $\mathcal{L}(\omega)$ is required to detect all solutions of the  Baouendi-Grushin Sch\"odinger equation from $\omega$, but that for $T_0>\mathcal{L}(\omega)$, no solution is localized enough to stay in $M\backslash \omega$ during the time interval $(0,T_0)$. Roughly speaking, the frequencies of order $\xi_y$ travel at speed $\sim\xi_y$, which is typical for a dispersion phenomenon. This picture is very different from the one for the wave equation (which we consider in this paper) for which no dispersion occurs. 

\smallskip

With similar ideas, in \cite{Letrouit20}, the interplay between the subellipticity effects measured by the non-holonomic order of the distribution $\mathcal{D}$ (see Section \ref{s:nilpotentization}) and the strength of dispersion of Schr\"odinger-type equations was investigated. More precisely, for $\Delta_\gamma=\partial_x^2+|x|^{2\gamma}\partial_y^2$ on $M=(-1,1)_x\times \mathbb{T}_y$, and for $s\in\mathbb{N}$, the observability properties of the Schr\"odinger-type equation $(i\partial_t-(-\Delta_\gamma)^s)u=0$ were shown to depend on the value $\kappa=2s/(\gamma+1)$. In particular it is proved that, for $\kappa<1$, observability fails for any time, which is consistent with the present result, and that for $\kappa=1$, observability holds only for sufficiently large times, which is consistent with the result of \cite{burq2019time}. The results of \cite{Letrouit20} are somehow Schr\"odinger analogues of the results of \cite{BCG14} which deal with a similar problem for the Baouendi-Grushin heat equation.

\paragraph{General bibliographical comments.} Control of subelliptic PDEs has attracted much attention in the last decade. Most results in the literature deal with subelliptic parabolic equations, either the  Baouendi-Grushin heat equation (\cite{koenig2017non}, \cite{duprez2018control}, \cite{beauchard20}) or the heat equation in the Heisenberg group (\cite{beauchard2017heat}, see also references therein). The paper \cite{burq2019time} is the first to deal with a subelliptic Schr\"odinger equation and the present work is the first to handle exact controllability of subelliptic wave equations.

\smallskip

A slightly different problem is the \emph{approximate} controllability of hypoelliptic PDEs, which has been studied in \cite{laurent2017tunneling} for hypoelliptic wave and heat equations. Approximate controllability is weaker than exact controllability, and it amounts to proving ``quantitative" unique continuation results for hypoelliptic operators. For the hypoelliptic wave equation, it is proved in \cite{laurent2017tunneling} that for $T>2 \sup_{x\in M}(\text{dist}(x,\omega))$ (here, dist is the sub-Riemannian distance), the observation of the solution on $(0,T)\times\omega$ determines the initial data, and therefore the whole solution.

\subsection{Organization of the paper} \label{s:strategy}

In Section \ref{s:constrgbsmain}, we construct exact solutions of the subelliptic wave equation \eqref{e:system} concentrating on any given normal geodesic. First, in Section \ref{s:gbapprox}, we show that, given any normal geodesic $t\mapsto x(t)$ which does not hit $\partial M$ in the time interval $(0,T)$, it is possible to construct a sequence $(v_k)_{k\in\mathbb{N}}$ of \emph{approximate solutions} of \eqref{e:system} whose energy concentrates along $t\mapsto x(t)$ during the time interval $(0,T)$ as $k\rightarrow +\infty$. By  ``approximate", we mean here that $\partial_{tt}^2v_k-\Delta v_k$ is small, but not necessarily exactly equal to $0$. In Section \ref{s:gbapprox}, we provide a first proof for this construction using the classical propagation of complex Lagrangian spaces. An other proof using a Gaussian beam approach is provided in Appendix \ref{a:approxgb}. Then, in Section \ref{s:gbexact}, using this sequence $(v_k)_{k\in\mathbb{N}}$, we explain how to construct a sequence $(u_k)_{k\in\mathbb{N}}$ of \emph{exact} solutions of $(\partial_{tt}^2-\Delta)u=0$ in $M$ with the same concentration property along the normal geodesic $t\mapsto x(t)$.

\smallskip

In Section \ref{s:spiralingmain}, we prove the existence of normal geodesics which spiral in $M$, spending an arbitrarily large time in $M\backslash \omega$. These normal geodesics generalize the example described in Section \ref{s:ideas} for the Heisenberg manifold with boundary. The proof proceeds in two steps: first, we show that it is sufficient to prove the result in the so-called ``nilpotent case" (Section \ref{s:spiralinggeneral}), and then we prove it in the nilpotent case (Section \ref{s:spiralingnilpotent}).

\smallskip

In Section \ref{s:concl}, we use the results of Section \ref{s:constrgbsmain} and Section \ref{s:spiralingmain} to conclude the proof of Theorem \ref{t:main}. In Section \ref{s:contr}, we deduce Corollary \ref{c:contr} by a duality argument. Finally, in Section \ref{s:truncatedobs}, we prove Theorem \ref{t:truncatedobs}.

\paragraph{Acknowledgments.} I warmly thank my PhD advisor Emmanuel Tr\'elat for mentioning this problem to me, for his constant support and his numerous suggestions during the preparation of this paper. Many thanks also to Andrei Agrachev who helped me correct a flaw in the proof of Proposition \ref{p:existencegeod}. I thank Yves Colin de Verdi\`ere, Luc Hillairet, Armand Koenig, Luca Rizzi, Clotilde Fermanian-Kammerer, Maciej Zworski, Fr\'ed\'eric Jean, Jean-Paul Gauthier, Matthieu L\'eautaud and Ludovic Sacchelli for interesting discussions related to this problem. Finally, I am very grateful to an anonymous referee whose questions and suggestions allowed me to considerably improve the readability of the present paper. I was partially supported by the grant ANR-15-CE40-0018 of the ANR (project SRGI).

\section{Gaussian beams along normal geodesics} \label{s:constrgbsmain}

\subsection{Construction of sequences of approximate solutions} \label{s:gbapprox}
We consider a solution $(x(t),\xi(t))_{t\in [0,T]}$ of \eqref{e:bicarac} on $M$. We shall describe the construction of solutions of 
\begin{equation} \label{e:sRwave}
\partial_{tt}^2u-\Delta u=0
\end{equation}
on $[0,T]\times M$ with energy
\begin{equation*}
E(u(t,\cdot)):=\frac12\left(\|\partial_tu(t,\cdot)\|_{L^2(M,\mu)}^2+\|u(t,\cdot)\|_{\mathcal{H}}^2\right)
\end{equation*}
concentrated along $x(t)$ for $t\in [0,T]$. The following proposition, which is inspired by \cite{ralston1982gaussian} and \cite{macia2002lack}, shows that it is possible, at least for approximate solutions of \eqref{e:sRwave}.

\begin{proposition} \label{p:approxgb}
Fix $T>0$ and let $(x(t),\xi(t))_{t\in [0,T]}$ be a solution of \eqref{e:bicarac} (in particular $g^*(x(0),\xi(0))=1/4$) which does not hit the boundary $\partial M$ in the time-interval $(0,T)$. Then there exist $a_0,\psi\in C^2((0,T)\times M)$ such that, setting, for $k\in\mathbb{N}$,
\begin{equation*}
v_k(t,x)=k^{\frac{n}{4}-1}a_0(t,x)e^{ik\psi(t,x)}
\end{equation*}
the following properties hold:
\begin{itemize}
\item $v_k$ is an approximate solution of \eqref{e:sRwave}, meaning that
\begin{equation} \label{e:boundedenergy}
\|\partial_{tt}^2v_k-\Delta v_k\|_{L^1((0,T);L^2(M))}\leq Ck^{-\frac12}.
\end{equation}
\item The energy of $v_k$ is bounded below with respect to $k$ and $t\in [0,T]$: 
\begin{equation} \label{e:convenergyforv}
\exists A>0, \forall t\in [0,T],\quad  \liminf_{\substack{k\rightarrow +\infty}}E(v_k(t,\cdot))\geq A.
\end{equation}
\item The energy of $v_k$ is small off $x(t)$: for any $t\in [0,T]$, we fix $V_t$ an open subset of $M$ for the initial topology of $M$, containing $x(t)$, so that the mapping $t\mapsto V_t$ is continuous ($V_t$ is chosen sufficiently small so that this makes sense in a chart). Then
\begin{equation} \label{e:decreaseenergyforv}
\sup_{t\in [0,T]}\int_{M\backslash V_t} \left(|\partial_tv_k(t,x)|^2+\sum_{j=1}^m(X_jv_k(t,x))^2\right)d\mu(x)\underset{k\rightarrow +\infty}{\rightarrow} 0.
\end{equation}
\end{itemize}
\end{proposition}

\begin{remark}
The construction of approximate solutions such as the ones provided by Proposition \ref{p:approxgb} is usually done for strictly hyperbolic operators, that is operators with a principal symbol $p_m$ of order $m$ such that the polynomial $f(s)=p_m(t,q,s,\xi)$ has $m$ distinct real roots when $\xi\neq 0$ (see for example \cite{ralston1982gaussian}). The operator $\partial_{tt}^2-\Delta$ is not strictly hyperbolic because $g^*$ is degenerate, but our proof shows that the same construction may be adapted without difficulty to this operator \emph{along normal bicharacteristics}. This is due to the fact that along normal bicharacteristics, $\partial_{tt}^2-\Delta$ is indeed strictly hyperbolic (or equivalently, $\Delta$ is elliptic). It was already noted by \cite{ralston1982gaussian} that the construction of Gaussian beams could be done for more general operators than strictly hyperbolic ones, and that the differences between the strictly hyperbolic case and more general cases arise while dealing with propagation of singularities. Also, in \cite[Chapter 24.2]{hormander2007analysis}, it was noticed that ``since only microlocal properties of $p_2$ are important, it is easy to see that hyperbolicity may be replaced by $\nabla_\xi p_2\neq 0$''.
\end{remark}

Hereafter we provide two proofs of Proposition \ref{p:approxgb}. The first proof is short and is actually quite straightforward for readers acquainted with the theory of propagation of complex Lagrangian spaces, once one has noticed that the solutions of \eqref{e:bicarac} which we consider live in the \emph{elliptic part} of the principal symbol of $-\Delta$. For the sake of completeness, and because this also has its own interest, we provide in Appendix \ref{a:approxgb} a second proof, longer but more elementary and accessible without any knowledge of complex Lagrangian spaces; it relies on the construction of Gaussian beams in the subelliptic context. The two proofs follow parallel paths, and indeed, the computations which are only sketched in the first proof are written in full details in the second proof, given in Appendix \ref{a:approxgb}.

\begin{proof}[First proof of Proposition \ref{p:approxgb}]
The construction of Gaussian beams, or more generally of a WKB approximation, is related to the transport of complex Lagrangian spaces along bicharacteristics, as reported for example in \cite[Chapter 24.2]{hormander2007analysis} and \cite[Volume I, Part I, Chapter 1.2]{ivrii2019microlocal}. Our proof follows the lines of \cite[pages 426-428]{hormander2007analysis}.

\smallskip

A usual way to solve (at least approximately) evolution equations of the form
\begin{equation} \label{e:semicl}
Pu=0
\end{equation}
where $P$ is a hyperbolic second order differential operator with real principal symbol and $C^\infty$ coefficients is to search for oscillatory solutions 
\begin{equation} \label{e:ansatz}
v_k(x)=k^{\frac{n}{4}-1}a_0(x)e^{ik\psi(x)}.
\end{equation}
 In this expression as in the rest of the proof, we suppress the time variable $t$. Thus, we use $x=(x_0,x_1,\ldots,x_n)$ where $x_0=t$ in the earlier notations, and we set $x'=(x_1,\ldots,x_n)$. Similarly, we take the notation $\xi=(\xi_0,\xi_1,\ldots,\xi_n)$ where $\xi_0=\tau$ previously, and $\xi'=(\xi_1,\ldots,\xi_n)$. The bicharacteristics are parametrized by $s$ as in \eqref{e:bicarac1}, and without loss of generality, we only consider bicharacteristics with $x(0)=0$ at $s=0$, which implies in particular $x_0(s)=s$ because of our choice $\tau^2(s)=g^*(x(s),\xi(s))=1/4$. 
 
 \smallskip
 
Taking charts of $M$, we can assume $M\subset\R^n$. The precise argument for reducing to this case is written at the end of Appendix \ref{a:approxgb}. Also, in the sequel, $P=\partial_{tt}^2-\Delta$.

\smallskip

Plugging the Ansatz \eqref{e:ansatz} into \eqref{e:semicl}, we get
\begin{equation} \label{e:opondesvk1}
Pv_k=(k^{\frac{n}{4}+1}A_1+k^{\frac{n}{4}}A_2+k^{\frac{n}{4}-1}A_3)e^{ik\psi}
\end{equation}
with 
\begin{align}
A_1(x)&=p_2\left(x,\nabla\psi(x)\right)a_0(x) \nonumber \\
A_2(x)&=La_0(x)\label{e:A2} \nonumber\\
A_3(x)&=\partial_{tt}^2a_0(x)-\Delta a_0(x).\nonumber
\end{align}
and $L$ is a transport operator given by 
\begin{equation} \label{e:La01}
La_0=\frac{1}{i}\sum_{j=0}^n\frac{\partial p_2}{\partial \xi_j}\left(x,\nabla\psi(x)\right)\frac{\d a_0}{\d x_j}+\frac{1}{2i}\left(\sum_{j,k=0}^n\frac{\partial^2p_2}{\partial\xi_j\partial\xi_k}\left(x,\nabla\psi(x)\right)\frac{\partial^2\psi}{\partial x_j\partial x_k}\right)a_0.
\end{equation}
In order for $v_k$ to be an approximate solution of $P$, we are first led to cancel the higher order term in \eqref{e:opondesvk1}, i.e.,
\begin{equation} \label{e:cancelfirstterm}
f(x):=p_2(x,\nabla\psi(x))=0
\end{equation}
which we solve for initial conditions 
\begin{equation} \label{e:initialpsi}
\psi(0,x')=\psi_0(x'), \qquad \nabla\psi_0(0)=\xi'(0) \quad \text{and}\quad \psi_0(0)=0
\end{equation}
 (i.e., we fix such a $\psi_0$, and then we solve \eqref{e:cancelfirstterm} for $\psi$). Indeed, it will be sufficient for our purpose for \eqref{e:cancelfirstterm} to be verified at second order along the curve $x(s)$, i.e., $D^\alpha_xf(x(s))=0$ for any $|\alpha|\leq 2$ and any $s$. For that, we first notice that the choice $\nabla\psi(x(s))=\xi(s)$ ensures that \eqref{e:cancelfirstterm} holds at orders $0$ and $1$ along the curve $s\mapsto x(s)$ (see Appendix \ref{a:approxgb} for detailed computations). Now, we explain how to choose $D^2\psi(x(s))$ adequately in order for \eqref{e:cancelfirstterm} to hold at order $2$.
 
 \smallskip

 We use the decomposition of $p_2$ into
\begin{equation*}
p_2(x_0,x',\xi_0,\xi')=-(\xi_0-r(x',\xi'))(\xi_0+r(x',\xi'))+R(x',\xi')
\end{equation*}
where $r=\sqrt{g^*}$ in a conic neighborhood of $(0,\xi(0))$. Note that $\sqrt{g^*}$ is smooth in small conic neighborhoods of $(0,\xi(0))$ since $g^*(0,\xi(0))=1/4\neq0$. Indeed, $g^*$ is elliptic along the whole bicharacteristic since $g^*(x(t),\xi(t))=1/4$ is preserved by the bicharacteristic flow. The rest term $R(x',\xi')$ is smooth and microlocally supported far from the bicharacteristic, i.e., $R(x',\xi')=0$ for any $(x',\xi')\in T^*M$ in a conic neighborhood of $(x'(s),\xi'(s))$ for $s\in [0,T]$.

\smallskip

We consider the bicharacteristic $\gamma_+$ starting at $(0,0,r(0,\xi'(0)),\xi'(0))$ and the bicharacteristic $\gamma_-$ starting at $(0,0,-r(0,\xi'(0)),\xi'(0))$.

\smallskip

We denote by $\Phi^{\pm}(x_0,y',\eta')$ the solution of the Hamilton equations with Hamiltonian $H_\pm(x_0,x',\xi')=\xi_0\mp r(x',\xi')$ and initial datum $(x',\xi')=(y',\eta')$ at $x_0=0$. In other words, $\Phi^\pm(x_0,y',\eta')=e^{x_0\vec{H}_{\pm}}(0,y',\eta')$.  Then, for any $s$, $\Phi(s,\cdot)$ is well-defined and symplectic from a neighborhood of $(0,\xi'(0))$ to a neighborhood of $H_\pm(s,0,\xi'(0))$.

\smallskip

The solution $\psi(s,\cdot)$ of \eqref{e:cancelfirstterm} and \eqref{e:initialpsi} is equal to $0$ on $\gamma_\pm$ and $\nabla\psi(s,\cdot)$ is obtained by the transport of the values of $\nabla\psi_0$ by $\Phi^\pm(s,\cdot)$. In other words, to compute $\nabla\psi(s,\cdot)$, one transports the Lagrangian sub-space $\Lambda_0=\{(x',\nabla\psi_0(x'))\}$ along the Hamiltonian flow $\vec{H}_\pm$ during a time $s$, which yields $\Lambda_s\subset T^*M$, and then, if possible, one writes $\Lambda_s$ under the form $\{(x',\nabla_{x'}\psi(s,x'))\}$, which gives $\nabla_{x'}\psi(s,x')$. The trouble is that the solution is only local in time:  when $x'\mapsto \pi(\Phi^\pm(s,x',\nabla\psi_0(x')))$ ceases to be a diffeomorphism (conjugate point), where $\pi:T^*M\rightarrow M$ is the canonical projection, we see that the process described above does not work (appearance of caustics). In the language of Lagrangian spaces, $\Lambda_0=\{(x',\nabla\psi_0(x'))\}\subset T^*M$ is a Lagrangian subspace and, since $\Phi^\pm(s,\cdot)$ is a symplectomorphism, $\Lambda_s=\Phi^\pm(s,\Lambda_0)$ is Lagrangian as well. If $\pi_{|\Lambda_s}$ is a local diffeomorphism, one can locally describe $\Lambda_s$ by  $\Lambda_s=\{(x',\nabla_{x'}\psi(s,x'))\}\subset T^*M$ for some function $\psi(s,\cdot)$, but blow-up happens when $\text{rank}(d\pi_{|\Lambda_s})<n$ (classical conjugate point theory), and such a $\psi(s,\cdot)$ may not exist.

\smallskip

However, if the phase $\psi_0$ is complex, quadratic, and satisfies the condition $\text{Im}(D^2\psi_0)>0$, where $D^2\psi_0$ denotes the Hessian, no blow-up happens, and the solution is global in time. Let us explain why. Indeed, $\Lambda_0=\{(x',\nabla\psi_0(x'))\}$ then lives in the complexification of the tangent space $T^*M$, which may be thought of as $\mathbb{C}^{2(n+1)}$. We take coordinates $(y,\eta)$ on $T^*\R^{n+1}$ or $T^*\mathbb{C}^{n+1}$ and we consider the symplectic forms defined by $\sigma=\sum dy_j\wedge d\eta_j$ and $\sigma_{\mathbb{C}}=\sum dy_j\wedge \overline{d{\eta_j}}$.

\smallskip

Because of the condition $\text{Im}(D^2\psi_0)>0$, $\Lambda_0$ is called a ``strictly positive Lagrangian space" (see \cite[Definition 21.5.5]{hormander2007analysis}), meaning that $i\sigma_{\mathbb{C}}(v,v)>0$ for $v$ in the tangent space to $\Lambda_0$. For any $s$, the symplectic forms $\sigma$ and $\sigma_{\mathbb{C}}$ are preserved by $\Phi(s,\cdot)$, meaning that $\Phi(s,\cdot)_*\sigma=\sigma$ and $\Phi(s,\cdot)_*\sigma_\C=\sigma_\C$, therefore $\sigma=0$ on the tangent space to $\Lambda_s$, and $i\sigma_{\mathbb{C}}(v,v)>0$ for $v$ tangent to $\Lambda_s$. It precisely means that $\Lambda_s$ is also a strictly positive Lagrangian space. Then, by \cite[Proposition 21.5.9]{hormander2007analysis}, we know that there exists $\psi(s,\cdot)$ complex and quadratic with $\text{Im}(D^2\psi(s,\cdot))>0$ such that $\Lambda_s=\{(x',\nabla_{x'}\psi(s,x'))\}$ (to apply  \cite[Proposition 21.5.9]{hormander2007analysis}, recall that for $\varphi(x')=\frac12(Ax',x')$, there holds $\nabla \varphi(x')=Ax'$). In other words, the key point in using complex phases is that strictly positive Lagrangian spaces are parametrized by complex quadratic phases $\varphi$ with $\text{Im}(D^2\varphi)>0$, whereas real Lagrangian spaces were not parametrized by real phases (see explanations above). This parametrization is a diffeomorphism from the Grassmannian of strictly positive Lagrangian spaces to the space of complex quadratic phases with $\varphi$ with $\text{Im}(D^2\varphi)>0$. Hence, the phase
\begin{equation*}
\psi(s,y')=\nabla_{x'}\psi(x(s))\cdot(y'-x'(s))+\frac12 (y'-x'(s))\cdot D_{x'}^2\psi(s,x'(s))(y'-x'(s))
\end{equation*}
for $s\in [0,T]$ and $y'\in\R^n$ is smooth and for this choice, \eqref{e:cancelfirstterm} is satisfied at second order along $s\mapsto x(s)$ (the rest $R(x',\xi')$ plays no role since it vanishes in a neighborhood of $s\mapsto x(s)$).

\smallskip

Then, we note that $A_2$ vanishes along the bicharacteristic if and only if $La_0(x(s))=0$ (see also \cite[Equation (24.2.9)]{hormander2007analysis}). According to \eqref{e:La01}, this turns out to be a linear transport equation on $a_0(x(s))$, with leading coefficient $\nabla_\xi p_2(x(s),\xi(s))$ different from $0$. Given $a\neq0$ at $(t=0,x'=x'(0))$, this transport equation has a solution $a_0(x(s))$ with initial datum $a$, and, by Cauchy uniqueness, $a_0(x(s))\neq 0$ for any $s$. We can choose $a_0$ in a smooth (and arbitrary) way outside the bicharacteristic. We choose it to vanish outside a small neighborhood of this bicharacteristic, so that no boundary effect happens.

\smallskip

With these choices of $\psi$ and $a_0$, the bound \eqref{e:boundedenergy} then follows from the following result whose proof is given in \cite[Lemma 2.8]{ralston1982gaussian}.
\begin{lemma}\label{l:oscint}
Let $c(x)$ be a function on $\R^{n+1}$ which vanishes at order $S-1$ on a curve $\Gamma$ for some $S\geq 1$. Suppose that $\text{Supp } c\cap \{|x_0|\leq T\}$ is compact and that $\text{Im }\psi(x)\geq ad(x)^2$ on this set for some constant $a>0$, where $d(x)$ denotes the distance from the point $x\in\R^{d+1}$ to the curve $\Gamma$. Then there exists a constant $C$ such that 
\begin{equation*}
\int_{|x_0|\leq T}\left| c(x)e^{ik\psi(x)}\right|^2dx\leq Ck^{-S-n/2}.
\end{equation*}
\end{lemma}
Let us now sketch the end of the proof, which is given in Appendix \ref{a:approxgb} in full details. We apply Lemma \ref{l:oscint} to $S=3$, $c=A_1$ and to $S=1$, $c=A_2$, and we get
\begin{equation*}
\|\partial_{tt}^2v_k-\Delta v_k\|_{L^1(0,T;L^2(M))}\leq C(k^{-\frac12}+k^{-\frac12}+k^{-1}),
\end{equation*}
which implies \eqref{e:boundedenergy}.
The bounds \eqref{e:convenergyforv} and \eqref{e:decreaseenergyforv} follow from the facts that $\text{Im}(D^2\psi(s,\cdot))>0$ and $v_k(x)=k^{\frac{n}{4}-1}a_0(x)e^{ik\psi(x)}$.
\end{proof}

\begin{remark}
An interesting question would be to understand the delocalization properties of the Gaussian beams constructed along normal geodesics in Proposition \ref{p:approxgb}. Compared with the usual Riemannian case done for example in \cite{ralston1982gaussian}, there is a new phenomenon in the sub-Riemannian case since the normal geodesic $x(t)$  (or, more precisely, its lift to $S^*M$) may approach the characteristic manifold $\Sigma=\{g^*=0\}$ which is the set of directions in which $\Delta$ is not elliptic. In finite time $T$ as in our case, the lift of the normal geodesic remains far from $\Sigma$, but it may happen as $T\rightarrow +\infty$ that it goes closer and closer to $\Sigma$. The question is then to understand the link between the delocalization properties of the Gaussian beams constructed along such a normal geodesic, and notably the interplay between the time $T$ and the semi-classical parameter $1/k$.
\end{remark}

\subsection{Construction of sequences of exact solutions in $M$} \label{s:gbexact}

In this section, using the approximate solutions of Proposition \ref{s:gbapprox}, we construct \textit{exact} solutions of \eqref{e:sRwave} whose energy concentrates along a given normal geodesic of $M$ which does not meet the boundary $\partial M$ during the time interval $[0,T]$.

\begin{proposition} \label{p:exactgb}
Let $(x(t),\xi(t))_{t\in [0,T]}$ be a solution of \eqref{e:bicarac} in $M$ (in particular $g^*(x(0),\xi(0))=1/4$) which does not meet $\partial M$. Let $\theta\in C_c^\infty([0,T]\times M)$ with $\theta(t,\cdot)\equiv 1$ in a neighborhood of $x(t)$ and such that the support of $\theta(t,\cdot)$ stays at positive distance of $\partial M$. 

Suppose $(v_k)_{k\in\mathbb{N}}$ is constructed along $x(t)$ as in Proposition \ref{p:approxgb} and $u_k$ is the solution of the Cauchy problem
\begin{equation} 
\left\lbrace \begin{array}{l}
(\partial_{tt}^2-\Delta)u_k=0 \text{  \ in } (0,T)\times M, \nonumber \\
u_{k}=0 \text{ \ in } (0,T)\times \partial M, \nonumber \\
u_{k|t=0}=(\theta v_k)_{|t=0}, \  \partial_tu_{k|t=0}=[\partial_t(\theta v_{k})]_{|t=0}. \nonumber
\end{array}\right.
\end{equation}
Then:
\begin{itemize}
\item The energy of $u_k$ is bounded below with respect to $k$ and $t\in [0,T]$:
\begin{equation} \label{e:convenergyforu2}
\exists A>0, \forall t\in [0,T],\quad  \liminf_{\substack{k\rightarrow +\infty}}E(u_k(t,\cdot))\geq A.
\end{equation}
\item The energy of $u_k$ is small off $x(t)$: for any $t\in [0,T]$, we fix $V_t$ an open subset of $M$ for the initial topology of $M$, containing $x(t)$, so that the mapping $t\mapsto V_t$ is continuous ($V_t$ is chosen sufficiently small so that this makes sense in a chart). Then
\begin{equation} \label{e:decreaseenergy2}
\sup_{t\in [0,T]}\int_{M\backslash V_t} \left(|\partial_tu_k(t,x)|^2+\sum_{j=1}^m (X_ju_k(t,x))^2\right)d\mu(x)\underset{k\rightarrow +\infty}{\rightarrow} 0.
\end{equation}
\end{itemize}
\end{proposition}

\begin{proof}[Proof of Proposition \ref{p:exactgb}]
Set $h_k=(\partial_{tt}^2-\Delta)(\theta v_k)$. We consider $w_k$ the solution of the Cauchy problem
\begin{equation} \label{e:cauchyu}
\left\lbrace \begin{array}{l}
(\partial_{tt}^2-\Delta)w_k=h_k \text{\ in } (0,T)\times M, \\
w_k=0  \text{ \ in } (0,T)\times \partial M,\\
(w_{k|t=0},\partial_tw_{k|t=0})=(0,0)\,.
\end{array}\right.
\end{equation}
Differentiating $E(w_k(t,\cdot))$ and using Gronwall's lemma, we get the energy inequality
\begin{equation*}
\sup_{t\in [0,T]} E(w_k(t,\cdot))\leq C\left(E(w_k(0,\cdot))+\|h_k\|_{L^1(0,T;L^2(M))}\right).
\end{equation*}
Therefore, using \eqref{e:boundedenergy}, we get $\sup_{t\in [0,T]} E(w_k(t,\cdot))\leq Ck^{-1}$. Since $u_k=\theta v_k-w_k$, we obtain that  
\begin{equation*}
\underset{k\rightarrow +\infty}{\lim}E(u_k(t,\cdot))=\underset{k\rightarrow +\infty}{\lim}E((\theta v_k)(t,\cdot))=\underset{k\rightarrow +\infty}{\lim}E(v_k(t,\cdot))
\end{equation*}
for every $t\in[0,T]$ where the last equality comes from the fact that $\theta$ and its derivatives are bounded and $\|v_k\|_{L^2}\leq Ck^{-1}$ when $k\rightarrow +\infty$. Using \eqref{e:convenergyforv}, we conclude that \eqref{e:convenergyforu2} holds. 

\smallskip

To prove \eqref{e:decreaseenergy2}, we observe similarly that 
\begin{align}
&\sup_{t\in [0,T]}\int_{M\backslash V_t} \left(|\partial_tu_k(t,x)|^2+\sum_{j=1}^m (X_ju_k(t,x))^2\right)d\mu(x) \nonumber \\
\leq \  & C\sup_{t\in [0,T]}\left(\int_{M\backslash V_t} \left(|\partial_tv_k(t,x)|^2+\sum_{j=1}^m (X_jv_k(t,x))^2\right)d\mu(x)\right)+Ck^{-\frac12} \nonumber \\
\longrightarrow & \  0  \nonumber
\end{align}
as $k\rightarrow +\infty$, according to \eqref{e:decreaseenergyforv}. It concludes the proof of Proposition \ref{p:exactgb}.
\end{proof}

\section{Existence of spiraling normal geodesics} \label{s:spiralingmain}
The goal of this section is to prove the following proposition, which is the second building block of the proof of Theorem \ref{t:main}, after the construction of localized solutions of the subelliptic wave equation \eqref{e:system} done in Section \ref{s:constrgbsmain}. 

\smallskip

We say that $X_1,\ldots,X_m$ satisfies the property $\mathbf{(P)}$ at $q\in M$ if the following holds:

\smallskip

$\mathbf{(P)}$ \ \emph{For any open neighborhood $V$ of $q$, for any $T_0>0$, there exists a non-stationary normal geodesic $t\mapsto x(t)$, traveled at speed $1$, such that $x(t)\in V$ for any $t\in [0,T_0]$.}

\begin{proposition} \label{p:existencegeod}
At any point $q\in M$ such that there exist $1\leq i,j\leq m$ with $[X_i,X_j](q)\notin \mathcal{D}_q$, property $\mathbf{(P)}$ holds.
\end{proposition}

In Section \ref{s:nilpotentization}, we define the so-called nilpotent approximations $\widehat{X}^q_1,\ldots,\widehat{X}^q_m$ at a point $q\in M$, which are first-order approximations of $X_1,\ldots,X_m$ at $q\in M$ such that the associated Lie algebra $\text{Lie}(\widehat{X}^q_1,\ldots,\widehat{X}^q_m)$ is nilpotent. Roughly, we have $\widehat{X}_i^q\approx X_i(q)$, but low order terms of $X_i(q)$ are not taken into account for defining $\widehat{X}_i^q$, so that the high order brackets of the $\widehat{X}_i^q$ vanish (which is not generally the case for the $X_i$). These nilpotent approximations are good local approximations of the vector fields $X_1,\ldots,X_m$, and their study is much simpler.

\smallskip

The proof of Proposition \ref{p:existencegeod} splits into two steps: first, we show that it is sufficient to prove the result in the nilpotent case (Section \ref{s:spiralinggeneral}), then we handle this simpler case (Section \ref{s:spiralingnilpotent}).

\subsection{Nilpotent approximation} \label{s:nilpotentization}
In this section, we recall the construction of the nilpotent approximations $\widehat{X}^q_1,\ldots,\widehat{X}^q_m$. The definitions we give are classical, and the reader can refer to \cite[Chapter 10]{agrachev2019} and \cite[Chapter 2]{jean2014control} for more material on this section. This construction is related to the notion of tangent space in the Gromov-Hausdorff sense of a sub-Riemannian structure $(M,\mathcal{D},g)$ at a point $q\in M$; the tangent space is defined intrinsically (meaning that it does not depend on a choice of coordinates or of local frame) as an equivalence class under the action of sub-Riemannian isometries (see \cite{bellaiche1996tangent}, \cite{jean2014control}).

\smallskip

\textbf{Sub-Riemannian flag.} We define the sub-Riemannian flag as follows: we set $\mathcal{D}^0=\{ 0\}$, $\mathcal{D}^1=\mathcal{D}$, and, for any $j\geq 1$, $\mathcal{D}^{j+1}=\mathcal{D}^j+[\mathcal{D},\mathcal{D}^j]$. For any point $q\in M$, it defines a flag
\begin{equation*}
\{0\}=\mathcal{D}_q^0\subset \mathcal{D}_q^1\subset \ldots \subset \mathcal{D}_q^{r-1} \varsubsetneq \mathcal{D}_q^{r(q)} =T_qM.
\end{equation*}
The integer $r(q)$ is called the non-holonomic order of $\mathcal{D}$ at $q$, and it is equal to $2$ everywhere in the Heisenberg manifold for example. Note that it depends on $q$, see Example \ref{ex:exempleBaouendi} in Section \ref{s:mainresult} (the Baouendi-Grushin example).

\smallskip
For $0\leq i \leq r(q)$, we set $n_i(q)=\dim \mathcal{D}_q^i$, and the sequence $(n_i(q))_{0\leq i\leq r(q)}$ is called the growth vector at point $q$. We set $\mathcal{Q}(q)=\sum_{i=1}^{r(q)}i(n_i(q)-n_{i-1}(q))$, which is generically the Hausdorff dimension of the metric space given by the sub-Riemannian distance on $M$ (see \cite{mitchell85}). Finally, we define the non-decreasing sequence of weights $w_i(q)$ for $1\leq i\leq n$ as follows. Given any $1\leq i\leq n$, there exists a unique $1\leq j\leq n$ such that $n_{j-1}(q)+1\leq i\leq n_j(q)$. We set $w_i(q)=j$. For example, for any $q$ in the Heisenberg manifold, $w_1(q)=w_2(q)=1$ and $w_3(q)=2$: indeed, the coordinates $x_1$ and $x_2$ have ``weight $1$'', while the coordinate $x_3$ has ``weight $2$'' since $\partial_{x_3}$ requires a bracket to be generated.

\smallskip

\textbf{Regular and singular points.} We say that $q\in M$ is regular if the growth vector $(n_i(q'))_{0\leq i\leq r(q')}$ at $q'$ is constant for $q'$ in a neighborhood of $q$. Otherwise, $q$ is said to be singular. If any point $q\in M$ is regular, we say that the structure is equiregular. For example, the Heisenberg manifold is equiregular, but not the Baouendi-Grushin example.

\smallskip

\textbf{Non-holonomic orders.} The non-holonomic order of a smooth germ of function is given by the formula
\begin{equation*}
\ord_q(f)=\min\{ s\in\mathbb{N} : \exists i_1,\ldots,i_s\in \{1,\ldots,m\} \text{ such that } (X_{i_1}\ldots X_{i_s}f)(q)\neq 0\}
\end{equation*}
where we adopt the convention that $\min\emptyset = +\infty$.

\smallskip

 The non-holonomic order of  a smooth germ of vector field $X$ at $q$, denoted by $\ord_q(X)$, is the real number defined by
\begin{equation*}
\ord_q(X)=\sup\{\sigma\in\R : \ord_q(Xf)\geq \sigma+\ord_q(f), \ \ \forall f \in C^\infty(q)\}.
\end{equation*} 
For example, there holds $\ord_q([X,Y])\geq \ord_q(X)+\ord_q(Y)$ and $\ord_{q}(fX)\geq \ord_q(f)+\ord_q(X)$. As a consequence, every $X$ which has the property that $X(q')\in\mathcal{D}^i_{q'}$ for any $q'$ in a neighborhood of $q$ is of non-holonomic order $\geq -i$.

\smallskip

\textbf{Privileged coordinates.} Locally around $q\in M$, it is possible to define a set of so-called ``privileged coordinates" of $M$ (see \cite{bellaiche1996tangent}).

\smallskip

A family $(Z_1,\ldots,Z_n)$ of $n$ vector fields is said to be adapted to the sub-Riemannian flag at $q$ if it is a frame of $T_qM$ at $q$ and if $Z_i(q)\in \mathcal{D}_q^{w_i(q)}$ for any $i\in \{1,\ldots,n\}$. In other words, for any $i\in\{1,\ldots,r(q)\}$, the vectors $Z_1, \ldots, Z_{n_i(q)}$ at $q$ span $\mathcal{D}_q^i$.

\smallskip

A system of privileged coordinates at $q$ is a system of local coordinates $(x_1,\ldots,x_n)$ such that 
\begin{equation}\label{e:nonholpriv}
\ord_q(x_i)=w_i, \qquad \text{ for $1\leq i\leq n$.}
\end{equation}
 In particular, for privileged coordinates, we have $\partial_{x_i}\in \mathcal{D}_q^{w_i(q)}\backslash \mathcal{D}_q^{w_i(q)-1}$ at $q$, meaning that privileged coordinates are adapted to the flag.

\smallskip

\textbf{Example: exponential coordinates of the second kind.} Choose an adapted frame $(Z_1,\ldots,Z_n)$ at $q$. It is proved in \cite[Appendix B]{jean2014control} that the inverse of the local diffeomorphism
\begin{equation*}
(x_1,\ldots,x_n)\mapsto \exp(x_1Z_1)\circ \cdots \circ \exp(x_n Z_n)(q)
\end{equation*}
defines privileged coordinates at $q$, called exponential coordinates of the second kind.

\smallskip

\textbf{Dilations.} We consider a chart of privileged coordinates at $q$ given by a smooth mapping $\psi_q : U\rightarrow \R^n$, where $U$ is a neighborhood of $q$ in $M$, with $\psi_q(q)=0$. For every $\varepsilon \in \mathbb{R}\backslash \{0\}$, we consider the dilation $\delta_\e:\R^n\rightarrow \R^n$ defined by
\begin{equation*}
\delta_\varepsilon(x)=(\varepsilon^{w_i(q)}x_1,\ldots,\varepsilon^{w_n(q)}x_n)
\end{equation*}
for every $x=(x_1,\ldots,x_n)$. A dilation $\delta_\e$ acts also on functions and vector fields on $\R^n$ by pull-back: $\delta_\e^*f=f\circ \delta_\e$ and $\delta_\e^*X$ is the vector field such that $(\delta_\e^*X)(\delta_\e^*f)=\delta_\e^*(Xf)$ for any $f\in C^1(\R^n)$. In particular, for any vector field $X$ of non-holonomic order $k$, there holds $\delta_\e^*X=\e^{-k}X$.

\smallskip

\textbf{Nilpotent approximation.} Fix a system of privileged coordinates $(x_1,\ldots,x_n)$ at $q$. Given a sequence of integers $\alpha=(\alpha_1,\ldots,\alpha_n)$, we define the weighted degree of $x^\alpha=x_1^{\alpha_1}\ldots x_n^{\alpha_n}$ to be $w(\alpha)=w_1(q)\alpha_1+\ldots+w_n(q)\alpha_n$. Coming back to the vector fields $X_1,\ldots,X_m$, we can write the Taylor expansion
\begin{equation}\label{e:Taylorexpansionvf}
X_i(x)\sim \sum_{\alpha,j}a_{\alpha,j}x^\alpha \partial_{x_j}.
\end{equation}
Since $X_i\in\mathcal{D}$, its non-holonomic order is necessarily $-1$, hence there holds $w(\alpha)\geq w_j(q)-1$ if $a_{\alpha,j}\neq 0$. Therefore, we may write $X_i$ as a formal series
\begin{equation*}
X_i=X_i^{(-1)}+X_i^{(0)}+X_i^{(1)}+\ldots
\end{equation*}
where $X_i^{(s)}$ is a homogeneous vector field of degree $s$, meaning that
\begin{equation*}
\delta_\varepsilon^*(\psi_q)_*X^{(s)}_i=\varepsilon^{s}(\psi_q)_*X^{(s)}_i.
\end{equation*}
We set $\widehat{X}^q_i=(\psi_q)_*X_i^{(-1)}$ for $1\leq i\leq m$. Then $\widehat{X}^q_i$ is homogeneous of degree $-1$ with respect to dilations, i.e., $\delta_\varepsilon^* \widehat{X}^q_i=\varepsilon^{-1}\widehat{X}^q_i$ for any $\varepsilon\neq 0$. Each $\widehat{X}^q_i$ may be seen as a vector field on $\R^n$ thanks to the coordinates $(x_1,\ldots,x_n)$. Moreover,
\begin{equation*}
\widehat{X}_i^q=\lim_{\e\rightarrow 0} \e\delta_{\e}^*(\psi_q)_*X_i
\end{equation*}
in $C^\infty$ topology: all derivatives uniformly converge on compact subsets. For $\e>0$ small enough we have
\begin{equation*}
X_i^\e:= \e\delta_{\e}^*(\psi_q)_*X_i=\widehat{X}_i^q+\e R_i^\e
\end{equation*}
where $R_i^\e$ depends smoothly on $\e$ for the $C^\infty$ topology (see also \cite[Lemma 10.58]{agrachev2019}). An important property is that $(\widehat{X}^q_1,\ldots,\widehat{X}^q_m)$ generates a nilpotent Lie algebra of step $r(q)$ (see \cite[Proposition 2.3]{jean2014control}). 

\smallskip

 The nilpotent approximation of $X_1,\ldots,X_m$ at $q$ is then defined as $\widehat{M}^q\simeq\R^n$ endowed with the vector fields $\widehat{X}_1^q,\ldots,\widehat{X}^q_m$. It is important to note that the nilpotent approximation depends on the initial choice of privileged coordinates. For an explicit example of computation of nilpotent approximation, see \cite[Example 2.8]{jean2014control}.

\subsection{Reduction to the nilpotent case} \label{s:spiralinggeneral}
In this section, we show the following 
\begin{lemma}\label{l:reduction}
Let $X_1,\ldots,X_m$ be smooth vector fields on $M$ satisfying  H\"ormander's condition, and let $q\in M$. If the property $\mathbf{(P)}$ holds at point $0\in \R^n$ for the nilpotent approximation $\widehat{X}^q_1,\ldots,\widehat{X}^q_m$, then the property $\mathbf{(P)}$ holds at point $q$ for $X_1,\ldots,X_m$.
\end{lemma}
Note that the above lemma is true for any nilpotent approximation $\widehat{X}^q_1,\ldots,\widehat{X}^q_m$ at $q$, i.e., for any choice of privileged coordinates  (see Section \ref{s:nilpotentization}).

\begin{proof}[Proof of Lemma \ref{l:reduction}] We use the notation $h_Z$ for the momentum map associated with the vector field $Z$ (see Section \ref{s:normal}).
We use the notations of Section \ref{s:nilpotentization}, in particular the coordinate chart $\psi_q$. 

\smallskip

 We set  $Y_i=(\psi_q)_*X_i$ and $X_i^\e=\e\delta_\e^*Y_i$ which is a vector field on $\R^n$. Recall that
\begin{equation*}
X_i^\e=\widehat{X}_i^q+\e R_i^\e
\end{equation*}
where $R_i^\e$ depends smoothly on $\e$ for the $C^\infty$ topology. Therefore, using the homogeneity of $\widehat{X}_i^q$, we get, for any $\varepsilon >0$,
\begin{equation} \label{e:diffYi}
Y_i=\frac{1}{\e}(\delta_\e)_*X_i^\e=\frac{1}{\e}(\delta_\e)_*(\widehat{X}_i^q+\e R_i^\e)=\widehat{X}_i^q+(\delta_\e)_*R_i^\e.
\end{equation}


The vector field $(\delta_\e)_*R_i^\e(x)$ does not depend on $\varepsilon$ and has a size which tends uniformly to $0$ as $x\rightarrow0\in \widehat{M}^q\simeq\R^n$.
 Recall that the Hamiltonian $\widehat{H}$ associated to the vector fields $\widehat{X}_i^q$ is given by
\begin{equation*}
\widehat{H}=\sum_{i=1}^mh_{\widehat{X}^q_i}^2.
\end{equation*}
Similarly, we set
\begin{equation*}
H=\sum_{i=1}^m h_{Y_i}^2.
\end{equation*}
 We note that \eqref{e:diffYi} gives
\begin{equation*}
h_{Y_i}=h_{\widehat{X}_i^q}+h_{(\delta_\e)_*R_i^\e}.
\end{equation*}
Hence 
\begin{equation} \label{e:perturbHamiltnilpo}
\vec{H}=2\sum_{i=1}^mh_{Y_i}\vec{h}_{Y_i}=\vec{\widehat{H}}+\vec{\Theta},
\end{equation}
where $\vec{\Theta}$ is a smooth vector field on $T^*\R^n$ such that 
\begin{equation} \label{e:petitTheta}
\|(d\pi\circ \vec{\Theta})(x,\xi)\|\leq C\|x\|
\end{equation}
 when $\|x\|\rightarrow 0$ (independently of $\xi$) where $\pi:T^*\R^n\rightarrow \R^n$ is the canonical projection. This last point comes from the smooth dependence of $R_i^\e$ on $\e$ for the $C^\infty$ topology (uniform convergence of all derivatives on compact subsets of $\R^n$). 
 
 \smallskip

Given the projection of an integral curve $c(\cdot)$  of $\vec{H}$, we denote by $\widehat{c}(\cdot)$ the projection of the integral curve of $\vec{\widehat{H}}$ with same initial covector. Combining \eqref{e:perturbHamiltnilpo} and \eqref{e:petitTheta}, and using Gronwall's lemma, we obtain the following result: 
\begin{center}
Fix $T_0>0$. For any neighborhood $V$ of $0$ in $\R^n$, there exists another neighborhood $V'$ of $0$ such that if $c_{|[0,T_0]}\subset V'$, then $\widehat{c}_{|[0,T_0]}\subset V$.
\end{center}

Therefore, if the property $\mathbf{(P)}$ holds at $0\in\R^n$ for $\widehat{X}_1^q,\ldots,\widehat{X}^q_m$, then it holds also at $0\in\R^n$ for the vector fields $Y_1,\ldots,Y_m$.

\smallskip

Using that $X_i=\psi_q^*Y_i$, we can pull back the result to $M$ and obtain that the property $\mathbf{(P)}$ holds at point $q$ for $X_1,\ldots,X_m$, which concludes the proof of Proposition \ref{p:existencegeod}.
\end{proof}

Thanks to Lemma \ref{l:reduction}, it is sufficient to prove the property $\mathbf{(P)}$ under the additional assumption that
\begin{equation} \label{e:assump}
\text{$M\subset \R^n$ and $\text{Lie}(X_1,\ldots,X_m)$ is nilpotent.}
\end{equation}
In all the sequel, we assume that this is the case.

\subsection{End of the proof of Proposition \ref{p:existencegeod}} \label{s:spiralingnilpotent}
Let us finish the proof of Proposition \ref{p:existencegeod}. Our ideas are inspired by \cite[Section 6]{agrachev2001}. 

\paragraph{First step: reduction to the constant Goh matrix case.} We consider an adapted frame $Y_1,\ldots,Y_n$ at $q$. We take exponential coordinates of the second kind at $q$: we consider the inverse $\psi_q$ of the diffeomorphism
$$
(x_1,\ldots,x_n)\mapsto \exp(x_1Y_1)\ldots\exp(x_nY_{n})(q).
$$
Then we write the Taylor expansion \eqref{e:Taylorexpansionvf} of $X_1,\ldots,X_m$ in these coordinates. Thanks to Lemma \ref{l:reduction}, we can assume that all terms in these Taylor expansions have non-holonomic order $-1$. We denote by $\xi_i$ the dual variable of $x_i$. We use the notations $n_1, n_2, \ldots$ introduced in Section \ref{s:nilpotentization}, and we make a strong use of \eqref{e:nonholpriv}.

\smallskip

\emph{Claim 1.} If a normal geodesic $(x(t),\xi(t))_{t\in\R}$ has initial momentum satisfying $\xi_k(0)=0$ for any $k\geq n_2+1$, then $\dot{\xi}_k\equiv0$ for any $k\geq n_1+1$, and in particular $\xi_k\equiv 0$ for any $k\geq n_2+1$.
\begin{proof} 
We write 
$$
X_j(x)=\sum_{i=1}^n a_{ij}(x)\partial_{x_i}, \qquad j=1,\ldots,m
$$
where the $a_{ij}$ are homogeneous polynomials. We have
\begin{equation}\label{e:g*redgoh}
g^*(x,\xi)=\sum_{j=1}^m\left(\sum_{i=1}^n a_{ij}(x)\xi_i\right)^2.
\end{equation}
Let $k\geq n_2+1$, which means that $x_k$ has non-holonomic order $\geq 3$. If $a_{ij}(x)$ depends on $x_k$, then necessarily $i\geq n_3+1$, since $a_{ij}(x)\partial_{x_i}$ has non-holonomic order $-1$. Thus, writing explicitly $\dot{\xi}_k=-\frac{\partial g^*}{\partial x_k}$ thanks to \eqref{e:g*redgoh}, there is in front of each term a factor $\xi_i$  for some $i$ which is in particular $\geq n_2+1$. By Cauchy uniqueness, we deduce that $\xi_k\equiv0$ for any $k\geq n_2+1$. 

\smallskip

Now, let $k\geq n_1+1$, which means that $x_k$ has non-holonomic order $\geq 2$. If $a_{ij}(x)$ depends on $x_k$, then necessarily $i\geq n_2+1$, since $a_{ij}(x)\partial_{x_i}$ has non-holonomic order $-1$. Thus, writing explicitly $\dot{\xi}_k=-\frac{\partial g^*}{\partial x_k}$ thanks to \eqref{e:g*redgoh}, there is in front of each term a factor $\xi_i$  for some $i$ which is $\geq n_2+1$. It is null by the previous conclusion, hence $\dot{\xi}_k\equiv 0$.
\end{proof}

\smallskip

The previous claim will help us reducing the complexity of the vector fields $X_i$ once again (after the first reduction provided by Lemma \ref{l:reduction}). Let us consider, for any $1\leq j\leq m$, the vector field
\begin{equation}\label{e:reducedvf}
X_j^{\rm red}=\sum_{i=1}^{n_2} a_{ij}(x)\partial_{x_i}
\end{equation}
where the sum is taken only up to $n_2$. We also consider the reduced Hamiltonian on $T^*M$
$$g^*_{\rm red}=\sum_{j=1}^m h_{X_j^{\rm red}}^2.$$

\smallskip

\emph{Claim 2.} If $X_1^{\rm red},\ldots,X_m^{\rm red}$ satisfy Property $\mathbf{(P)}$ at $q$, then $X_1,\ldots,X_m$ satisfy Property $\mathbf{(P)}$ at $q$.
\begin{proof}
Let us assume that $X_1^{\rm red},\ldots,X_m^{\rm red}$ satisfy Property $\mathbf{(P)}$ at $q$. Let $T_0>0$ and let $(x^{{\rm red}, \e}(0),\xi^{{\rm red}, \e}(0))$ be initial data for the Hamiltonian system associated to $g^*_{\rm red}$ which yield speed $1$ normal geodesics $(x^{{\rm red}, \e}(t),\xi^{{\rm red}, \e}(t))$ such that $x^{{\rm red}, \e}(t)\rightarrow q$ uniformly over $(0,T_0)$ as $\e\rightarrow 0$. 

\smallskip

We can assume without loss of generality that $\xi_i^{{\rm red}, \e}(0)=0$ for any $i\geq n_2+1$, since these momenta (preserved under the reduced Hamiltonian evolution) do not change the projection $x^{{\rm red}, \e}(t)$ of the normal geodesic. We consider $(x^\e(0),\xi^\e(0))=(x^{{\rm red}, \e}(0),\xi^{{\rm red}, \e}(0))$ as initial data for the (non-reduced) Hamiltonian evolution associated to $g^*$. Then we notice that $\xi_k^\e\equiv 0$ for $k\geq n_2+1$ thanks to Claim 1. It follows that when $i\leq n_2$, we have $x_i^\e(t)=x_i^{{\rm red}, \e}(t)$, i.e., the coordinate $x_i$ is the same for the reduced and the non-reduced Hamiltonian evolution.

\smallskip

Finally, we take $k$ such that $n_2+1\leq k\leq n_3$. Since $g^*$ is given by \eqref{e:g*redgoh}, we have
\begin{equation}\label{e:xkdot}
\dot{x}_k^\e=\frac{\partial g^*}{\partial \xi_k}=2\sum_{j=1}^m a_{kj}(x^\e)\left(\sum_{i=1}^n a_{ij}(x^\e)\xi^\e_i\right).
\end{equation}
But $a_{kj}$ has necessarily non-holonomic order $2$ since $\partial_{x_k}$ has non-holonomic order $-3$. Thus, $a_{kj}(x)$ is a non-constant homogeneous polynomial in $x_1,\ldots,x_{n_2}$. Since $x_1^\e,\ldots,x_{n_2}^\e$ converge to $q$ uniformly over $(0,T_0)$ as $\e\rightarrow 0$, it is also the case of $x_k^\e$ according to \eqref{e:xkdot}, noticing that 
$$
\left|\sum_{i=1}^n a_{ij}(x^\e)\xi^\e_i\right|\leq (g^*)^{1/2}=1/2
$$
for any $j$. In other words, $x_{n_2+1}^\e,\ldots,x_{n_3}^\e$ also converge to $q$ uniformly over $(0,T_0)$ as $\e\rightarrow 0$.

\smallskip

We can repeat this argument successively for $k\in \{n_3+1,\ldots,n_4\}$, $k\in\{n_4+1,\ldots,n_5\}$, etc, and we finally obtain the result: for any $1\leq k\leq n$, $x_k^\e$ converges to $q$ uniformly over $(0,T_0)$ as $\e\rightarrow 0$.
\end{proof}

\smallskip

Thanks to the previous claim, we are now reduced to prove Proposition  \ref{p:existencegeod} for the vector fields $X_1^{\rm red},\ldots,X_m^{\rm red}$. In order to keep notations as simple as possible, we simplify these notations into $X_1,\ldots,X_m$, i.e., we drop the upper notation ``red''. Also, without loss of generality we assume that $q=0$.

\smallskip

If we choose our normal geodesics so that $x(0)=0$, then $x_i\equiv 0$ for any $i\geq n_2+1$ thanks to \eqref{e:reducedvf}. In other words, we forget the coordinates $x_{n_2+1},\ldots,x_n$ in the sequel, since they all vanish.\footnote{Note that this is the case only because we are now working with the reduced Hamiltonian evolution; otherwise, under the original Hamiltonian evolution associated to \eqref{e:g*redgoh}, the $x_i$ (for $i\geq n_2+1$) remain small according to Claim 2, but do not necessarily vanish.}

%

\paragraph{Second step: conclusion of the proof.}

Now, we write the normal extremal system in its ``control'' form. We refer the reader to \cite[Chapter 4]{agrachev2019}. We have
\begin{equation}\label{e:controlform}
\dot{x}(t)=\sum_{i=1}^m u_i(t)X_i(x(t)),
\end{equation}
where the $u_i$ are the controls, explicitly given by
\begin{equation}\label{e:uexplicit}
u_i(t)=2h_{X_i}(x(t),\xi(t))
\end{equation} 
since $(x(t),\xi(t))=e^{t\vec{g}^*}(0,\xi_0)$. Thanks to \eqref{e:reducedvf}, we rewrite \eqref{e:controlform} as
\begin{equation}\label{e:diffeqxt}
\dot{x}(t)=F(x(t))u(t), 
\end{equation}
where $F=(a_{ij})$, which has size $n_2\times m$, and $u=\;^t(u_1,\ldots,u_m)$. Differentiating \eqref{e:uexplicit}, we have the complementary equation
$$
\dot{u}(t)=G(x(t),\xi(t))u(t)
$$
where $G$ is the Goh matrix
$$
G=(2\{h_{X_j},h_{X_i}\})_{1\leq i,j\leq m}
$$
(it differs from the usual Gox matrix by a factor $-2$ due to the absence of factor $\frac12$ in the Hamiltonian $g^*$ in our notations).

Let us prove that  $G(t)$ is constant in $t$. Fix $1\leq j,j'\leq m$. We notice that in \eqref{e:reducedvf},  $a_{ij}$ is a constant (independent of $x$) as soon as $1\leq i\leq n_1$ since $\partial_{x_i}$ has weight $-1$.  This implies that 
\begin{equation}\label{e:spandirectXi}
\text{$[X_j,X_{j'}]$ is spanned by the vector fields  $\partial_{x_{n_1+1}},\partial_{x_{n_1+2}},\ldots,\partial_{x_{n_2}}$.}
\end{equation}
  Putting this into the relation $\{h_{X_j},h_{X_{j'}}\}=h_{[X_j,X_{j'}]}$, and using that the dual variables $\xi_k$ for $n_1+1\leq k\leq n_2$ are preserved under the Hamiltonian evolution (due to Claim 1), we get that $G(t)\equiv G$ is constant in $t$. 

\smallskip

We know that $G\neq 0$ and that $G$ is antisymmetric. The whole control space $\R^m$ is the direct sum of the image of $G$ and the kernel of $G$, and $G$ is nondegenerate on its image.   We take $u_0$ in an invariant plane of $G$; in other words its projection on the kernel of $G$ vanishes (see Remark \ref{r:singular}). We denote by $\widetilde{G}$ the restriction of $G$ to this invariant plane. We also assume that $u_0$, decomposed as $u_0=(u_{01},\ldots,u_{0m})\in\R^m$, satisfies $\sum_{i=1}^m u_{0i}^2=1/4$. Then $u(t)=e^{t\widetilde{G}}u_0$ and since $e^{t\widetilde{G}}$ is an orthogonal matrix, we have $\|e^{t\widetilde{G}}u_0\|=\|u_0\|$. We have by integration by parts
\begin{align}
x(t)&=\int_0^tF(x(s))e^{s\widetilde{G}}u_0\,ds\nonumber\\
&=F(x(t))\widetilde{G}^{-1}(e^{t\widetilde{G}}-I)u_0-\int_0^t\frac{d}{ds}(F(x(s))\widetilde{G}^{-1}(e^{s\widetilde{G}}-I)u_0\,ds.\label{e:duhamel}
\end{align}

\smallskip

Let us now choose the initial data of our family of normal geodesics (indexed by $\e$). The starting point $x^\e(0)=0$ is the same for any $\e$, we only have to specify the initial covectors $\xi^\e=\xi^\e(0)\in T_0^*\R^n$. For any $i=1,\ldots,m$, we impose that
\begin{equation}\label{e:covyieldscontrol}
\langle \xi^\e, X_i\rangle=u_{0i}.
\end{equation}
It follows that $g^*(x(0),\xi^\e(0))=\sum_{i=1}^m u_{0i}^2=1/4$ for any $\e>0$. Now, we notice that $\text{Span}(X_1,\ldots,X_m)$ is in direct sum with the Span of the $[X_i,X_j]$ for $i,j$ running over $1,\ldots,m$ (this follows from \eqref{e:spandirectXi}). Fixing $G^0\neq 0$ an antisymmetric matrix and $\widetilde{G}^0$ its restriction to an invariant plane, we can specify, simultaneously to \eqref{e:covyieldscontrol}, that
$$
\langle \xi^\e, 2[X_j,X_i]\rangle = \varepsilon^{-1}G^0_{ij}.
$$
Then $x^\e(t)$ is given by \eqref{e:duhamel} applied with $\widetilde{G}=\e^{-1}\widetilde{G}^0$, which brings a factor $\e$ in front of \eqref{e:duhamel}.

\smallskip

Recall finally that the coefficients $a_{ij}$ which compose $F$ have non-holonomic order $0$ or $1$, thus they are degree $1$ (or constant) homogeneous polynomials in $x_1,\ldots,x_{n_1}$. Thus $\frac{d}{ds}(F(x(s))$ is a linear combination of $\dot{x}_i(s)$ which we can rewrite thanks to \eqref{e:diffeqxt} as a combination with bounded coefficients (since $\sum_{i=1}^m u_i^2=1/4$) of the $x_i(s)$. Hence, applying the Gronwall lemma in \eqref{e:duhamel}, we get $\|x^\e(t)\|\leq C\varepsilon$, which concludes the proof.
\begin{remark} \label{r:singular}
Let us explain why we choose $u_0$ to be in an invariant plane of $G$. If the projection of $u_0$ to the kernel of $G$ is nonzero then
the primitive of the exponential of $e^{\frac t\varepsilon G_0}u_0$ contains a linear term that does not depend on $\varepsilon$. Then the corresponding trajectory follows a singular curve (see \cite[Chapter 4]{agrachev2019} for a definition). This means, we find normal geodesics which spiral around a singular curve and do not remain close to their initial point over $(0,T_0)$, although their initial covector is ``high in the cylinder bundle $U^*M$''. For example, for the Hamiltonian $\xi_1^2+(\xi_2+x_1^2\xi_3)^2$  associated to the ``Martinet'' vector fields $X_1=\partial_{x_1}$, $X_2=\partial_{x_2}+x_1^2\partial_{x_3}$ in $\R^3$, there exist normal geodesics which spiral around the singular curve $(t,0,0)$.
\end{remark}

\begin{remark}
The normal geodesics constructed above lose their optimality quickly, in the sense that their first conjugate point and their cut-point are close to $q$. 
\end{remark}

\section{Proofs}

\subsection{Proof of Theorem \ref{t:main}} \label{s:concl}

In this section, we conclude the proof of Theorem \ref{t:main}.

\smallskip

Fix a point $q$ in the interior of $M\setminus \omega$ and $1\leq i,j\leq m$ such that $[X_i,X_j](q)\notin \mathcal{D}_q$. Fix also an open neighborhood $V$ of $q$ in $M$ such that $V\subset M\backslash\omega$. Fix $V'$ an open neighborhood of $q$ in $M$ such that $\overline{V'}\subset V$, and fix also $T_0>0$.

\smallskip

As already explained in Section \ref{s:ideas},  to conclude the proof of Theorem \ref{t:main}, we use Proposition \ref{p:exactgb} applied to the particular normal geodesics constructed in Proposition \ref{p:existencegeod}.

\smallskip

By Proposition \ref{p:existencegeod}, we know that there exists a normal geodesic $t\mapsto x(t)$ such that $x(t)\in V'$ for any $t\in (0,T_0)$. It is the projection of a bicharacteristic $(x(t),\xi(t))$ and since it is non-stationary and traveled at speed $1$, there holds $g^*(x(t),\xi(t))=1/4$. We denote by $(u_k)_{k\in\mathbb{N}}$ a sequence of solutions of \eqref{e:sRwave} as in Proposition \ref{p:exactgb} whose energy at time $t$ concentrates on $x(t)$ for $t\in (0,T_0)$. Because of \eqref{e:convenergyforu2}, we know that 
\begin{equation*}
\|(u_k(0),\partial_tu_k(0))\|_{\mathcal{H}\times L^2} \geq c>0
\end{equation*}
uniformly in $k$.

\smallskip

Therefore, in order to establish Theorem \ref{t:main}, it is sufficient to show that
\begin{equation} \label{e:suff}
\int_0^{T_0}\int_{\omega} |\partial_tu_k(t,x)|^2d\mu(x)dt\underset{k\rightarrow +\infty}{\rightarrow}0.
\end{equation}
Since $x(t)\in V'$ for any $t\in (0,T_0)$, we get that for $V_t$ chosen sufficiently small for any $t\in (0,T_0)$, the inclusion $V_t\subset V$ holds (see Proposition \ref{p:exactgb} for the definition of $V_t$). Combining this last remark with \eqref{e:decreaseenergy2}, we get \eqref{e:suff}, which concludes the proof of Theorem \ref{t:main}.


\subsection{Proof of Corollary \ref{c:contr}} \label{s:contr}

We endow the topological dual $\mathcal{H}(M)'$ with the norm $\|v\|_{\mathcal{H}(M)'}=\|(-\Delta)^{-1/2}v\|_{L^2(M)}$.

\smallskip

The following proposition is standard (see, e.g., \cite{tucsnak2009observation}, \cite{le2017geometric}).
\begin{lemma} \label{l:obslowerreg}
Let $T_0>0$, and $\omega\subset M$ be a measurable set. Then the following two observability properties are equivalent:

{\bf(P1):} There exists $C_{T_0}$ such that for any $(v_0,v_1)\in D((-\Delta)^{\frac12})\times L^2(M)$, the solution $v\in C^0(0,T_0; D((-\Delta)^{\frac12}))\cap C^1(0,T_0;L^2(M))$ of \eqref{e:system} satisfies 
\begin{equation} \label{e:strongobs2}
\int_0^{T_0} \int_\omega|\partial_t v(t,q)|^2d\mu(q)dt \geq C_{T_0} \|(v_0,v_1)\|_{\mathcal{H}(M)\times L^2(M)}.
\end{equation}

{\bf(P2):} There exists $C_{T_0}$ such that for any $(v_0,v_1)\in L^2(M)\times D((-\Delta)^{-\frac12})$, the solution $v\in C^0(0,T_0; L^2(M))\cap C^1(0,T_0;D((-\Delta)^{-\frac12}))$ of \eqref{e:system} satisfies 
\begin{equation} \label{e:obslower}
\int_0^{T_0} \int_\omega|v(t,q)|^2d\mu(q)dt \geq C_{T_0} \|(v_0,v_1)\|^2_{L^2\times \mathcal{H}(M)'}.
\end{equation}
\end{lemma}
\begin{proof}
Let us assume that (P2) holds. Let $u$ be a solution of \eqref{e:system} with initial conditions $(u_0,u_1)\in D((-\Delta)^{\frac12})\times L^2(M)$. We set $v=\partial_tu$, which is a solution of \eqref{e:system} with initial data $v_{|t=0}=u_1\in L^2(M)$ and $\partial_tv_{|t=0}=\Delta u_0\in D((-\Delta)^{-\frac12})$. Since $\|(v_0,v_1)\|_{L^2\times \mathcal{H}(M)'}=\|(u_1,\Delta u_0)\|_{L^2\times \mathcal{H}(M)'}=\|(u_0,u_1)\|_{\mathcal{H}(M)\times L^2}$, applying the observability inequality \eqref{e:obslower} to $v=\partial_tu$, we obtain \eqref{e:strongobs2}. The proof of the other implication is similar.
\end{proof}

Finally, using Theorem \ref{t:main}, Lemma \ref{l:obslowerreg} and the standard HUM method (\cite{lions1988controlabilite}), we get  Corollary \ref{c:contr}.

\subsection{Proof of Theorem \ref{t:truncatedobs}} \label{s:truncatedobs}

We consider the space of functions $u\in C^\infty([0,T]\times M_H)$ such that $\int_{M_H}u(t,\cdot)d\mu=0$ for any $t\in[0,T]$, and we denote by $\mathcal{H}_T$ its completion for the norm $\|\cdot\|_{\mathcal{H}_T}$ induced by the scalar product 
\begin{equation*}
(u,v)_{\mathcal{H}_T}=\int_0^T\int_{M_H} \left(\partial_tu\partial_tv+(X_1u)(X_1v)+(X_2u)(X_2v)\right) d\mu dt.
\end{equation*}
We consider also the topological dual $\mathcal{H}_0'$ of the space $\mathcal{H}_0$ (see Section \ref{s:truncatedobsstatement}).
\begin{lemma}\label{l:compactembedding}
The injections $\mathcal{H}_0\hookrightarrow L^2(M_H)$, $L^2(M_H) \hookrightarrow \mathcal{H}_0'$ and $\mathcal{H}_T\hookrightarrow L^2((0,T)\times M_H)$ are compact. 
\end{lemma}
\begin{proof}
Let $(\varphi_k)_{k\in\N}$ be an orthonormal basis of real eigenfunctions of $L^2(M_H)$, labeled with increasing eigenvalues $0=\lambda_0<\lambda_1\leq \ldots\leq \lambda_k \rightarrow +\infty$, so that $-\Delta_H \varphi_k=\lambda_k\varphi_k$. The fact that $\lambda_1>0$, which will be used in the sequel, can be proved as follows: if $-\Delta_H\varphi=0$ then $\int_{M_H} ((X_1\varphi)^2+(X_2\varphi)^2)\,d\mu=0$ and, since $\varphi\in C^\infty(M_H)$ by hypoelliptic regularity, we get $X_1\varphi(x)=X_2\varphi(x)=0$ for any $x\in M_H$. Hence, $[X_1,X_2]\varphi\equiv 0$, and alltogether, this proves that $\varphi$ is constant, hence $\lambda_1>0$.

\smallskip

We prove the last injection. Let $u\in\mathcal{H}_T$. Writing $u(t,\cdot)=\sum_{k=1}^{\infty} a_k(t)\varphi_k(\cdot)$ (note that there is no $0$-mode since $u(t,\cdot)$ has null average), we see that 
\begin{align*}
\|u\|_{\mathcal{H}_T}^2\geq (-\Delta_H u,u)_{L^2((0,T)\times M_H)}=\sum_{k=1}^\infty \lambda_k \|a_k\|_{L^2((0,T))}^2 &\geq \lambda_1\sum_{k=1}^\infty \|a_k\|_{L^2((0,T))}^2 \\
&= \lambda_1 \|u\|_{L^2((0,T)\times M_H)}^2,
\end{align*}
thus $\mathcal{H}_T$ imbeds continuously into $L^2((0,T)\times M_H)$. Then, using a classical subelliptic estimate (see \cite{hormander1967hypoelliptic} and \cite[Theorem 17]{rothschild1976hypoelliptic}), we know that there exists $C>0$ such that
\begin{equation*}\label{e:stein}
\|u\|_{H^{\frac12}((0,T)\times M_H)}\leq C(\|u\|_{L^2((0,T) \times M_H)}+\|u\|_{\mathcal{H}_T}).
\end{equation*} 
Together with the previous estimate, we obtain that for any $u\in\mathcal{H}_T$, $\|u\|_{H^{\frac12}((0,T)\times M_H)}\leq C\|u\|_{\mathcal{H}_T}$. Then, the result follows from the fact that the injection $H^{\frac12}((0,T)\times M_H)\hookrightarrow L^2((0,T)\times M_H)$ is compact.

\smallskip

The proof of the compact injection $\mathcal{H}_0\hookrightarrow L^2(M_H)$ is similar, and the compact injection $ L^2(M_H)\hookrightarrow \mathcal{H}_0'$ follows by duality.
\end{proof}

\begin{proof}[Proof of Theorem \ref{t:truncatedobs}]
In this proof, we use the notation $P=\partial_{tt}^2-\Delta_H$. For the sake of a contradiction, suppose that there exists a sequence $(u^k)_{k\in\mathbb{N}}$ of solutions of the wave equation such that $\|(u^k_0,u^k_1)\|_{\mathcal{H}\times L^2}=1$ for any $k\in\mathbb{N}$ and
\begin{equation} \label{e:twoconvweakobs}
\|(u_0^k,u_1^k)\|_{L^2\times \mathcal{H}_0'}\rightarrow 0,\quad \int_0^T|(\Op(a)\partial_tu^k,\partial_tu^k)_{L^2(M_H,\mu)}|dt\rightarrow 0
\end{equation}  
as $k\rightarrow +\infty$. Following the strategy of \cite{tartar1990h} and \cite{gerard1991mesures}, our goal is to associate a defect measure to the sequence $(u^k)_{k\in\mathbb{N}}$. Since the functional spaces involved in our result are unusual, we give the argument in detail. 

\smallskip

First, up to extraction of a subsequence which we omit, $(u^k_0,u^k_1)$ converges weakly in $\mathcal{H}_0\times L^2(M_H)$ and, using the first convergence in \eqref{e:twoconvweakobs} and the compact embedding $\mathcal{H}_0\times L^2(M_H)\hookrightarrow L^2(M_H)\times \mathcal{H}_0'$, we get that $(u^k_0,u^k_1)\rightharpoonup 0$ in $\mathcal{H}_0\times L^2_0$. Using the continuity of the solution with respect to the initial data, we obtain that $u^k\rightharpoonup 0$ weakly in $\mathcal{H}_T$. Using Lemma \ref{l:compactembedding}, we obtain $u^k\rightarrow 0$ strongly in $L^2((0,T)\times M_H)$.

\smallskip

Fix $B\in \Psi_{\phg}^0((0,T)\times M_H)$. We have 
\begin{align}
&(Bu^k,u^k)_{\mathcal{H}_T}\nonumber\\
&\quad =\int_0^T\int_{M_H} \left(\left(\partial_tBu^k\right)\left(\partial_tu^k\right)+\left(X_1Bu^k\right)\left(X_1u^k\right)+\left(X_2Bu^k\right)\left(X_2u^k\right)\right) d\mu(q)dt \nonumber\\
&\quad=\int_0^T\int_{M_H} \left(\left([\partial_t,B]u^k\right)\left(\partial_tu^k\right)+\left([X_1,B]u^k\right)\left(X_1u^k\right)+\left([X_2,B]u^k\right)\left(X_2u^k\right)\right) d\mu(q)dt  \nonumber \\
&\qquad +\int_0^T\int_{M_H} \left(\left(B\partial_t u^k\right)\left( \partial_tu^k\right)+\left(BX_1u^k\right) \left(X_1u^k\right)+\left(BX_2u^k\right) \left(X_2u^k\right)\right)d\mu(q)dt \label{e:bracketarg}
\end{align}
Since $[\partial_t,B]\in \Psi_{\phg}^0((0,T)\times M_H)$, $[X_j,B]\in \Psi_{\phg}^0((0,T)\times M_H)$ and $u^k\rightarrow 0$ strongly in $L^2((0,T)\times M_H)$, the first one of the two lines in \eqref{e:bracketarg} converges to $0$ as $k\rightarrow +\infty$. Moreover, the last line is bounded uniformly in $k$ since $B\in\Psi_{\phg}^0((0,T)\times M_H)$. Hence $(Bu^k,u^k)_{\mathcal{H}_T}$ is uniformly bounded. By a standard diagonal extraction argument (see \cite{gerard1991mesures} for example), there exists a subsequence, which we still denote by $(u^k)_{k\in\mathbb{N}}$ such that $(Bu^k,u^k)$ converges for any $B$ of principal symbol $b$ in a countable dense subset of $C_c^\infty((0,T)\times M_H)$. Moreover, the limit only depends on the principal symbol $b$, and not on the full symbol.

\smallskip

Let us now prove that
\begin{equation} \label{e:garding1}
\liminf_{k\rightarrow +\infty} \ (Bu^k,u^k)_{\mathcal{H}_T} \geq 0
\end{equation}
when $b\geq 0$. With a bracket argument as in \eqref{e:bracketarg}, we see that it is equivalent to proving that the liminf as $k\rightarrow +\infty$ of the quantity
\begin{equation} \label{e:garding2}
Q_k(B)= (B\partial_tu^k,\partial_tu^k)_{L^2}+(BX_1u^k,X_1u^k)_{L^2}+(BX_2u^k,X_2u^k)_{L^2}
\end{equation} 
is $\geq 0$. But there exists $B'\in \Psi_{\phg}^0((0,T)\times M_H)$ such that $B'-B\in\Psi_{\phg}^{-1}((0,T)\times M_H)$ and $B'$ is positive (this is the so-called Friedrichs quantization, see for example \cite[Chapter VII]{taylorpseudodifferential}). Then, $\liminf_{k\rightarrow +\infty} Q_k(B')\geq0$, and $Q_k(B'-B)\rightarrow 0$ since $(B'-B)\partial_t\in\Psi^0_\phg((0,T)\times M_H)$ and $u^k\rightarrow 0$ strongly in $L^2((0,T)\times M_H)$. It immediately implies that  \eqref{e:garding1} holds.

\smallskip

Therefore, setting $p=\sigma_p(P)$ and denoting by $\mathcal{C}(p)$ the characteristic manifold $\mathcal{C}(p)=\{p=0\}$, there exists a non-negative Radon measure $\nu$ on $S^*(\mathcal{C}(p))=\mathcal{C}(p)/(0,+\infty)$ such that
\begin{equation*}
(\Op(b)u^k,u^k)_{\mathcal{H}_T}\rightarrow \int_{S^*(\mathcal{C}(p))} bd\nu
\end{equation*}
for any $b\in S_{\phg}^0((0,T)\times M_H)$. 

\smallskip

Let $C\in \Psi_{\phg}^{-1}((0,T)\times M_H)$ of principal symbol $c$. We have $\vec{p}c=\{p,c\}\in S_{\phg}^0((0,T)\times M_H)$ and, for any $k\in \mathbb{N}$,  
\begin{equation} \label{e:comPC}
((CP-PC)u^k,u^k)_{\mathcal{H}_T} = (CPu^k,u^k)_{\mathcal{H}_T}-(Cu^k,Pu^k)_{\mathcal{H}_T}=0
\end{equation}
since $Pu^k=0$. To be fully rigorous, the identity of the previous line, which holds for any solution $u\in\mathcal{H}_T$ of the wave equation, is first proved for smooth initial data since $Pu\notin \mathcal{H}_T$ in general, and then extended to general solutions $u\in\mathcal{H}_T$. Taking principal symbols in \eqref{e:comPC}, we get $\langle \nu,\vec{p}c\rangle=0$. 

\smallskip

Therefore, denoting by $(\psi_s)_{s\in\R}$ the maximal solutions of
\begin{equation*}
\frac{d}{ds}\psi_s(\rho)=\vec{p}(\psi_s(\rho)), \qquad \rho\in T^*(\R\times M_H)
\end{equation*}
(see \eqref{e:bicarac1}), we get that, for any $s\in(0,T)$,
\begin{equation*}
0=\langle \nu,\vec{p}c\circ\psi_s\rangle=\langle \nu, \frac{d}{ds}c\circ\psi_s\rangle=\frac{d}{ds}\langle \nu,c\circ\psi_s\rangle
\end{equation*}
and hence 
\begin{equation}\label{e:propagmu2}
\langle \nu,c\rangle=\langle \nu,c\circ \psi_s\rangle.
\end{equation}
We note here that the precise homogeneity of $c$ (namely $c\in S_{\phg}^{-1}((0,T)\times M_H)$) does not matter since $\nu$ is a measure on the sphere bundle $S^*(\mathcal{C}(p))$. The identity \eqref{e:propagmu2} means that $\nu$ is invariant under the flow $\vec{p}$.

\smallskip

From the second convergence in \eqref{e:twoconvweakobs}, we can deduce that
\begin{equation}\label{e:nullitynu}
\nu=0 \text{  in  } S^*(\mathcal{C}(p)) \cap T^*((0,T)\times \text{Supp}(a)).
\end{equation}
The proof of this fact, which is standard (see for example \cite[Section 6.2]{burq2002controle}), is given in Appendix \ref{a:proofnullitynu}.

\smallskip

Let us prove that any normal geodesic of $M_H$ with momentum $\xi\in V_\varepsilon^c$ enters $\omega$ in time at most $\kappa\varepsilon^{-1}$ for some $\kappa>0$ which does not depend on $\varepsilon$. Indeed, the solutions of the bicharacteristic equations \eqref{e:bicharH} with $g^*=1/4$ and $\xi_3\neq0$ are given by
\begin{align*}
&x_1(t)=\frac{1}{2\xi_3}\cos(2\xi_3t+\phi)+\frac{\xi_2}{\xi_3}, \qquad x_2(t)=B-\frac{1}{2\xi_3}\sin(2\xi_3t+\phi)\\
&\qquad x_3(t)=C+\frac{t}{4\xi_3}+\frac{1}{16\xi_3^2}\sin(2(2\xi_3t+\phi))+\frac{\xi_2}{2\xi_3^2}\sin(2\xi_3t+\phi)
\end{align*}
where $B,C,\xi_2,\xi_3$ are constants. Since $\xi\in V_\varepsilon^c$ and $g^*=1/4$, there holds $\frac{1}{4|\xi_3|}\geq \frac{\varepsilon}{2}$. Hence, we can conclude using the expression for $x_3$ (whose derivative is roughly $(4|\xi_3|)^{-1}$) and the fact that $\omega=M_H\backslash B$ contains a horizontal strip. Note that if $\xi_3=0$, the expressions of $x_1(t), x_2(t), x_3(t)$ are much simpler and we can conclude similarly.

\smallskip
 
 Hence, together with \eqref{e:nullitynu}, the propagation property \eqref{e:propagmu2} implies that $\nu\equiv 0$. It follows that $\|u^k\|_{\mathcal{H}_T}\rightarrow 0$. By conservation of energy, it is a contradiction with the normalization $\|(u^k_0,u^k_1)\|_{\mathcal{H}\times L^2}=1$. Hence, \eqref{e:obstruncated} holds.
\end{proof}

\appendix

\section{Pseudodifferential calculus} \label{a:pseudo}
We denote by $\Omega$ an open set of a $d$-dimensional manifold (typically $d=n$ or $d=n+1$ with the notations of this paper) equipped with a smooth volume $\mu$. We denote by $q$ the variable in $\Omega$, typically $q=x$ or $q=(t,x)$ with our notations. 

\smallskip

Let $\omega_0=dp\wedge dq$ be the canonical symplectic form on $T^*\Omega$ written in canonical coordinates $(q,p)$. The Hamiltonian vector field $\vec{f}$ of a function $f\in C^\infty(T^*\Omega)$ is defined by the relation $$\omega_0(\vec{f},\cdot)=-df(\cdot).$$  In the coordinates $(q,p)$, it reads
\begin{equation*}
\vec{f}=\sum_{j=1}^d (\partial_{p_j}f) \partial_{q_j} -(\partial_{q_j} f)\partial_{p_j}.
\end{equation*}
In these coordinates, the Poisson bracket is 
\begin{equation*}
\{f,g\}=\omega_0(\vec{f},\vec{g})=\sum_{j=1}^d (\partial_{p_j}f) (\partial_{q_j}g) -(\partial_{q_j} f)(\partial_{p_j}g),
\end{equation*}
which is also equal to $\vec{f}g$ and $-\vec{g}f$.

\smallskip

Let $\pi:T^*\Omega\rightarrow \Omega$ be the canonical projection. We recall briefly some facts concerning pseudodifferential calculus, following \cite[Chapter 18]{hormander2007analysis}.

\smallskip

We denote by $S_\hom^m(T^*\Omega)$ the set of homogeneous symbols of degree $m$ with compact support in $\Omega$. We also write $S_\phg^m(T^*\Omega)$ the set of polyhomogeneous symbols of degree $m$ with compact support in $\Omega$. Hence, $a\in S_\phg^m(T^*\Omega)$ if $a\in C^\infty(T^*\Omega)$, $\pi(\supp(a))$ is a compact of $\Omega$, and there exist $a_j\in S^{m-j}_\hom(T^*\Omega)$ such that for all $N\in \mathbb{N}$, $a-\sum_{j=0}^N a_j\in S_\phg^{m-N-1}(T^*\Omega)$. We denote by $\Psi^m_\phg(T^*\Omega)$ the space of polyhomogeneous pseudodifferential operators of order $m$ on $\Omega$, with a compactly supported kernel in $\Omega\times\Omega$. For $A\in \Psi_\phg^m(\Omega)$, we denote by $\sigma_p(A)\in S^m_\phg(T^*\Omega)$ the principal symbol of $A$. The sub-principal symbol is characterized by the action of pseudodifferential operators on oscillating functions: if $A\in \Psi_\phg^m(\Omega)$ and $f(q)=b(q)e^{ikS(q)}$ with $b,S$ smooth and real-valued, then
\begin{equation*}
\int_\Omega A(f)\overline{f}d\mu=k^m\int_\Omega \left(\sigma_p(A)(q,S'(q))+\frac{1}{k}\sigma_{\text{sub}}(A)(q,S'(q))\right)|f(q)|^2d\mu(q)+O(k^{m-2}).
\end{equation*}

A quantization is a continuous linear mapping
\begin{equation*}
\Op:S^m_\phg(T^*\Omega)\rightarrow \Psi^m_\phg(\Omega)
\end{equation*}
satisfying $\sigma_p(\Op(a))=a$. An example of quantization is obtained by using partitions of unity and, locally, the Weyl quantization, which is given in local coordinates by
\begin{equation*}
\Op^W(a)f(q)=\frac{1}{(2\pi)^d}\int_{\R^d_{q'}\times \R^d_{p}}e^{i\langle q-q',p\rangle}a\left(\frac{q+q'}{2},p\right)f(q')dq'dp.
\end{equation*}

We have the following properties:
\begin{enumerate}
\item If $A\in \Psi^l_\phg(\Omega)$ and $B\in \Psi^m_\phg(\Omega)$, then $[A,B]\in\Psi^{l+m-1}_\phg(\Omega)$ and $\sigma_p([A,B])=\frac{1}{i}\{\sigma_p(a),\sigma_p(b)\}$.
\item If $X$ is a vector field on $\Omega$ and $X^*$ is its formal adjoint in $L^2(\Omega,\mu)$, then $X^*X\in \Psi^2_\phg(\Omega)$, $\sigma_p(X^*X)=h_X^2$ and $\sigma_{\text{sub}}(X^*X)=0$.
\item If $A\in \Psi_\phg^m(\Omega)$, then A maps continuously the space $H^s(\Omega)$ to the space $H^{s-m}(\Omega)$. 
\end{enumerate}

\section{Proof of Proposition \ref{p:approxgb}} \label{a:approxgb}
In this Appendix, we give a second proof of Proposition \ref{p:approxgb} written in a more elementary form than the one of Section \ref{s:gbapprox}. Let us first prove the result when $M\subset\R^n$, following the proof of \cite{ralston1982gaussian}. The general case is addressed at the end of this section.

\smallskip

As in the proof of Section \ref{s:gbapprox}, we suppress the time variable $t$. Thus we use $x=(x_0,x_1,\ldots,x_n)$ where $x_0=t$. Similarly, $\xi=(\xi_0,\xi_1,\ldots,\xi_n)$ where $\xi_0=\tau$ previously. Let $\Gamma$ be the curve given by $x(s)\in\R^{n+1}$. We insist on the fact that in the proof the bicharacteristics are parametrized  by $s$, as in \eqref{e:bicarac1}. We consider functions of the form 
\begin{equation*}
v_k(x)=k^{\frac{n}{4}-1}a_0(x)e^{ik\psi(x)}.
\end{equation*}

We would like to choose $\psi(x)$ such that for all $s\in \R$, $\psi(x(s))$ is real-valued and $\text{Im} \frac{\partial^2\psi}{\partial x_i\d x_j}(x(s))$ is positive definite on vectors orthogonal to $\dot{x}(s)$. Roughly speaking, $|e^{ik\psi(x)}|$ will then look like a Gaussian distribution on planes perpendicular to $\Gamma$ in $\R^{n+1}$.

\smallskip

We first observe that $\partial_{tt}^2v_k-\Delta v_k$ can be decomposed as
\begin{equation} \label{e:opondesvk}
\partial_{tt}^2v_k-\Delta v_k=(k^{\frac{n}{4}+1}A_1+k^{\frac{n}{4}}A_2+k^{\frac{n}{4}-1}A_3)e^{ik\psi}
\end{equation}
with 
\begin{align}
A_1(x)&=p_2\left(x,\nabla\psi(x)\right)a_0(x) \nonumber \\
A_2(x)&=La_0(x)\label{e:A2} \nonumber\\
A_3(x)&=\partial_{tt}^2a_0(x)-\Delta a_0(x).\nonumber
\end{align}
Here we have set
\begin{equation} \label{e:La0}
La_0=\frac{1}{i}\sum_{j=0}^n\frac{\partial p_2}{\partial \xi_j}\left(x,\nabla\psi(x)\right)\frac{\d a_0}{\d x_j}+\frac{1}{2i}\left(\sum_{j,k=0}^n\frac{\partial^2p_2}{\partial\xi_j\partial\xi_k}\left(x,\nabla\psi(x)\right)\frac{\partial^2\psi}{\partial x_j\partial x_k}\right)a_0
\end{equation}
(For general strictly hyperbolic operators, $L$ contains a term with the sub-principal symbol of the operator, but here it is null, see Appendix \ref{a:pseudo}.)

\smallskip

In what follows, we construct $a_0$ and $\psi$ so that $A_1(x)$ vanishes at order $2$ along $\Gamma$ and $A_2(x)$ vanishes at order $0$ along the same curve. We will then be able to use Lemma \ref{l:oscint} with $S=3$ and $S=1$ respectively.

\smallskip

\paragraph{Analysis of $A_1(x)$.} Our goal is to show that, if we choose $\psi$ adequately, we can make the quantity
\begin{equation} \label{e:forder0}
f(x)=p_2\left(x,\nabla\psi(x)\right)
\end{equation} 
vanish at order $2$ on $\Gamma$. For the vanishing at order $0$, we prescribe that $\psi$ satisfies $\nabla\psi(x(s))=\xi(s)$, and then $f(x(s))=0$ since $(x(s),\xi(s))$ is a null-bicharacteristic. Note that this is possible since $x(s)\neq x(s')$ for any $s\neq s'$, due to $\dot{x}_0(s)=1$ (bicharacteristics are traveled at speed $1$, see Section \ref{s:normal}). For the vanishing at order $1$, using \eqref{e:forder0} and \eqref{e:bicarac1}, we remark that for any $0\leq j\leq n$,
\begin{align}
\frac{\partial f}{\partial x_j}(x(s))&=\frac{\partial p_2}{\partial x_{j}}(x(s))+\sum_{k=0}^n\frac{\partial p_2}{\partial \xi_{k}}(x(s))\frac{\partial \psi}{\partial x_j \partial x_k}(x(s)) \nonumber \\
&=-\dot{\xi}_j(s)+\sum_{k=0}^n\dot{x}_k(s)\frac{\partial \psi}{\partial x_j \partial x_k}(x(s)) \label{e:firstorderM} \\
&=-\frac{d}{ds}\left(\frac{\d\psi}{\d x_j}(x(s))\right)+\sum_{k=0}^n\dot{x}_k(s)\frac{\partial \psi}{\partial x_j \partial x_k}(x(s))\nonumber \\
&=0. \nonumber
\end{align}
Therefore, $f$ vanishes automatically at order $1$ along $\Gamma$ (without making any particular choice for $\psi$): it just follows from \eqref{e:forder0} and the bicharacteristic equations \eqref{e:bicarac1}. But for $f(x)$ to vanish at order $2$ along $\Gamma$, it is required to choose a particular $\psi$. In the end, we will find that if $\psi$ is given by the formula \eqref{e:exprpsi} below, with $M$ being a solution of \eqref{e:systemM}, then  $f$ vanishes at order $2$ along $\Gamma$. Let us explain why.

\smallskip

Using the Einstein summation notation, we want that for any $0\leq i,j\leq n$, there holds
\begin{align}\label{e:eqsecondorder}
0&=\frac{\partial^2f}{\partial x_j\partial x_i}\nonumber \\
&=\frac{\partial^2p_2}{\partial x_j\partial x_i}+ \frac{\partial^2p_2}{\partial \xi_k\partial x_i}\frac{\partial^2\psi}{\partial x_j\partial x_k}+\frac{\partial^2p_2}{\partial x_j\partial \xi_k}\frac{\partial^2\psi}{\partial x_i\partial x_k}+\frac{\partial^2p_2}{\partial \xi_l\partial \xi_k}\frac{\partial^2\psi}{\partial x_i\partial x_k}\frac{\partial^2\psi}{\partial x_j\partial x_l}+\frac{\partial p_2}{\partial \xi_k}\frac{\partial^3\psi}{\partial x_j\partial x_k\partial x_i}\nonumber
\end{align}
along $\Gamma$. Introducing the matrices
\begin{align}
&(M(s))_{ij}=\frac{\partial^2\psi}{\partial x_i\partial x_j}(x(s)), &(A(s))_{ij}=\frac{\partial^2 p_2}{\partial x_i\partial x_j}(x(s),\xi(s)), \nonumber \\
&(B(s))_{ij}=\frac{\partial^2 p_2}{\partial \xi_i\partial x_j}(x(s),\xi(s)), &(C(s))_{ij}=\frac{\partial^2 p_2}{\partial \xi_i\partial \xi_j}(x(s),\xi(s)) \nonumber
\end{align}
this amounts to solving the matricial Riccati equation
\begin{equation} \label{e:systemM}
\frac{dM}{ds}+MCM+B^TM+MB+A=0
\end{equation}
on a finite-length time-interval. While solving \eqref{e:systemM}, we also require $M(s)$ to be symmetric, $\text{Im}(M(s))$ to be positive definite on the orthogonal complement of $\dot{x}(s)$, and $M(s)\dot{x}(s)=\dot{\xi}(s)$ to hold for all $s$ due to \eqref{e:firstorderM}.

\smallskip

Let $M_0$ be a symmetric $(n+1)\times (n+1)$ matrix with $\text{Im}(M_0)>0$ on the orthogonal complement of $\dot{x}(0)$ and $M_0\dot{x}(0)=\dot{\xi}(0)$ (in particular $\text{Im}(M_0)\dot{x}(0)=0$). It is shown in \cite{ralston1982gaussian} that there exists a global solution $M(s)$ on $[0,T]$ of \eqref{e:systemM} which satisfies all the above conditions and such that $M(0)=M_0$. The proof just requires that $A,C$ are symmetric, but does not need anything special about $p_2$ (in particular, it applies to our sub-Riemannian case where $p_2$ is degenerate). For the sake of completeness, we recall the proof here.

\smallskip

We consider $(Y(s),N(s))$ the matrix solution with initial data $(Y(0),N(0))=(\text{Id},M_0)$ (where $\text{Id}$ is the $(n+1)\times (n+1)$ identity matrix) to the linear system
\begin{equation} \label{e:systemyeta}
\left\lbrace \begin{array}{ll}
\dot{Y}=BY+CN\\
\dot{N}=-AY-B^TN.
\end{array}\right.
\end{equation}
We note that $(Y(s)\dot{x}(0),N(s)\dot{x}(0))$ then also solves \eqref{e:systemyeta}, with $Y$ and $N$ being this time vectorial. One can check that $(\dot{x}(s),\dot{\xi}(s))$ is the solution of the same linear system with same initial data, and therefore, for any $s\in\R$,
\begin{equation} \label{explsolYN}
\dot{x}(s)=Y(s)\dot{x}(0), \qquad \qquad \dot{\xi}(s)=N(s)\dot{x}(0).
\end{equation}

All the coefficients in \eqref{e:systemyeta} are real and $A$ and $C$ are symmetric, and it follows that the flow defined by \eqref{e:systemyeta} \emph{on vectors} preserves both the real symplectic form acting on pairs $(y,\eta)\in (\R^{n+1})^2$ and $(y',\eta')\in(\R^{n+1})^2$ given by
\begin{equation*}
\sigma((y,\eta),(y',\eta'))=y\cdot \eta'-\eta\cdot y'
\end{equation*}
and the complexified form $\sigma_{\mathbb{C}}((y,\eta),(y',\eta'))=\sigma((y,\eta),(\overline{y'},\overline{\eta'}))$ for $(y,\eta)\in (\mathbb{C}^{n+1})^2$ and $(y',\eta')\in (\mathbb{C}^{n+1})^2$. When we say that $\sigma_{\mathbb{C}}$ is invariant under \eqref{e:systemyeta}, it means that we allow complex vectorial initial data in \eqref{e:systemyeta}. 

\smallskip

Let us prove that $Y(s)$ is invertible for any $s$. Let $v\in\mathbb{C}^{n+1}$ and $s_0\in\R$ be such that $Y(s_0)v=0$. We set $y(s_0)=Y(s_0)v$ and $\eta(s_0)=N(s_0)v$ and consider $\chi(s_0)=(y(s_0),\eta(s_0))$. From the conservation of $\sigma_{\mathbb{C}}$, we get
\begin{align*}
0=\sigma_{\mathbb{C}}(\chi(s_0),\chi(s_0))=\sigma_{\C}(\chi(0),\chi(0))=v\cdot \overline{M_0v}-\overline{v}\cdot M_0v=-2i\overline{v}\cdot (\text{Im}(M_0))v.
\end{align*}
Since $\text{Im}(M_0)$ is positive definite on the orthogonal complement to $\dot{x}(0)$, there holds $v=\lambda \dot{x}(0)$ for some $\lambda\in\C$. Hence 
\begin{equation*}
0=Y(s_0)v=\lambda Y(s_0)\dot{x}(0)=\lambda \dot{x}(s_0)
\end{equation*}
where the last equality comes from \eqref{explsolYN}. Since $\dot{x}_0(s_0)=\frac{\partial p_2}{\partial\xi_0}(s_0)=-2\xi_0(s_0)=1$, there holds $\dot{x}(s_0)\neq0$, hence $\lambda=0$. It follows that $v=0$ and $Y(s_0)$ is invertible.

\smallskip

Now, for any $s\in\R$, we set
\begin{equation*}
M(s)=N(s)Y(s)^{-1}
\end{equation*}
which is a solution of \eqref{e:systemM} with $M(0)=M_0$. It verifies $M(s)\dot{x}(s)=\dot{\xi}(s)$ thanks to \eqref{explsolYN}. Moreover, it is symmetric: if we denote by $y^i(s)$ and $\eta^i(s)$ the column vectors of $Y$ and $N$, by preservation of $\sigma$, for any $0\leq i,j\leq n$, the quantity
\begin{equation*}
\sigma((y^i(s),\eta^i(s)),(y^j(s),\eta^j(s))=y^i(s)\cdot M(s)y^j(s)-y^j(s)\cdot M(s)y^i(s)
\end{equation*}
is equal to the same quantity at $s=0$, which is equal to $0$ since $M_0$ is symmetric.

\smallskip

Let us finally prove that for any $s\in\R$, $\text{Im}(M(s))$ is positive definite on the orthogonal complement of $\dot{x}(s)$. Let $y(s_0)\in\mathbb{C}^{n+1}$ be in the orthogonal complement of $\dot{x}(s_0)$. We decompose $y(s_0)$ on the column vectors of $Y(s_0)$:
\begin{equation*}
y(s_0)=\sum_{i=0}^n b_iy^i(s_0), \qquad b_i\in\C.
\end{equation*}
For $s\in\R$, we consider $y(s)=\sum_{i=0}^n b_iy^i(s)$ and we set $\chi(s)=\sum_{i=0}^n b_i(y^i(s),\eta^i(s))$. Then,
\begin{equation} \label{e:exprchis}
\sigma_\C(\chi(s),\chi(s))=-2i\overline{y(s)}\cdot \text{Im}(M(s))y(s).
\end{equation}
 By preservation of $\sigma_\C$ and using \eqref{e:exprchis}, we get that 
\begin{equation} \label{e:ImM}
\overline{y(s_0)}\cdot \text{Im}(M(s_0))y(s_0)=\overline{y(0)}\cdot \text{Im}(M_0)y(0).
\end{equation}
But $y(0)$ cannot be proportional to $\dot{x}(0)$ otherwise, using \eqref{explsolYN}, we would get that $y(s_0)$ is proportional to $\dot{x}(s_0)$. Hence, the right hand side in \eqref{e:ImM} is $>0$, which implies that $\text{Im}(M(s_0))$ is positive definite on the orthogonal complement to $\dot{x}(s_0)$.

\smallskip

Therefore, we found a choice for the second order derivatives of $\psi$ along $\Gamma$ which meets all our conditions. For $x=(t,x')\in\R\times\R^{n}$ and $s$ such that $t=t(s)$, we set
\begin{equation} \label{e:exprpsi}
\psi(x)=\xi'(s)\cdot (x'-x'(s))+\frac12 (x'-x'(s))\cdot M(s)(x'-x'(s)),
\end{equation} 
and for this choice of $\psi$, $f$ vanishes at order $2$ along $\Gamma$.

To sum up, as in the Riemannian (or ``strictly hyperbolic") case handled by Ralston in \cite{ralston1982gaussian}, the key observation is that the invariance of $\sigma$ and $\sigma_{\mathbb{C}}$ prevents the solutions of \eqref{e:systemM} with positive imaginary part on the orthogonal complement of $\dot{x}(0)$ to blowup.

\paragraph{Analysis of $A_2(x)$.} We note that $A_2$ vanishes along $\Gamma$ if and only if $La_0(x(s))=0$. According to \eqref{e:La0}, this turns out to be a linear transport equation on $a_0(x(s))$. Moreover, the coefficient of the first-order term, namely $\nabla_\xi p_2(x(s),\xi(s))$, is different from $0$. Therefore, given $a_0\neq0$ at $(t=0,x=x(0))$, this transport equation has a solution $a_0(x(s))$ with initial datum $a_0$, and, by Cauchy uniqueness, $a_0(x(s))\neq 0$ for any $s$. Note that we have prescribed $a_0$ only along $\Gamma$, and we may choose $a_0$ in a smooth (and arbitrary) way outside $\Gamma$. We choose it to vanish outside a small neighborhood of $\Gamma$.

\paragraph{Proof of \eqref{e:boundedenergy}.} We use \eqref{e:opondesvk} and we apply Lemma \ref{l:oscint} to $S=3$, $c=A_1$ and to $S=1$, $c=A_2$, and we get
\begin{equation*}
\|\partial_{tt}^2v_k-\Delta v_k\|_{L^1(0,T;L^2(M))}\leq C(k^{-\frac12}+k^{-\frac12}+k^{-1}),
\end{equation*}
which implies \eqref{e:boundedenergy}.

\paragraph{Proof of \eqref{e:convenergyforv}.} We first observe that  since $\text{Im}(M(s))$ is positive definite on the orthogonal complement of $\dot{x}(s)$ and continuous as a function of $s$, there exist $\alpha,C>0$ such that for any $t(s)\in [0,T]$ and any $x'\in M$,
\begin{equation*}
|\partial_tv_k(t(s),x')|^2+\sum_{j=1}^m|X_jv_k(t(s),x')|^2\geq \left(C|a_0(t(s),x')|^2k^{\frac{n}{2}}+O(k^{2(\frac{n}{2}-1)})\right)e^{-\alpha k d(x',x'(s))^2}
\end{equation*}
where $d(\cdot,\cdot)$ denotes the Euclidean distance in $\R^n$. We denote by $\ell_n$ the Lebesgue measure on $\R^n$. Using the observation that for any function $f$, 
\begin{equation}\label{e:generalintest}
\int_{M}f(x')e^{-\alpha kd(x',x'(s))^2}d\mu(x')\sim \frac{\pi^{n/2}}{k^{n/2}\sqrt{\alpha}}f(x'(s))\frac{d\mu}{d\ell_n}(x'(s))
\end{equation}
as $k\rightarrow +\infty$,  and the fact that $a_0(x(s))\neq 0$, we obtain \eqref{e:convenergyforv}.

\paragraph{Proof of \eqref{e:decreaseenergyforv}.} We observe that since $\text{Im}(M(s))$ is positive definite (uniformy in $s$) on the orthogonal complement of $\dot{x}(s)$, there exist $C,\alpha'>0$ such that for any $t\in[0,T]$, for any $x'\in M$, $|\partial_tv_k(t(s),x')|$ and $|X_j v_k(t(s),x')|$ are both bounded above by $C k^{\frac{n}{4}}e^{-\alpha' kd(x',x'(s))^2}$. Therefore
\begin{align} \label{e:ineqforenergy}
\int_{M\backslash V_{t(s)}}  \left(|\partial_tv_k(t(s),x')|^2+\sum_{j=1}^m|X_j v_k(t(s),x')|^2\right)d\mu(x')& \nonumber \\
\leq Ck^{n/2}\int_{M\backslash V_{t(s)}} e^{-2\alpha' kd(x',x'(s))^2}d\mu(x')& \nonumber \\
\leq Ck^{n/2}\int_{M\backslash V_{t(s)}} e^{-2\alpha' kd(x',x'(s))^2}d\ell_n(x') &+ \o(1)
\end{align}
where, in the last line, we used the fact that $|d\mu/d\ell_n|\leq C$ in a fixed compact subset of $M$ (since $\mu$ is a smooth volume), and the $\o(1)$ comes from the eventual blowup of $\mu$ at the boundary of $M$.

\smallskip

Now, $M\subset \R^n$, and there exists $r>0$ such that $B_d(x(s),r)\subset V_{t(s)}$ for any $s$ such that $t(s)\in (0,T)$, where $d(\cdot,\cdot)$ still denotes the Euclidean distance in $\R^n$. Therefore, we bound above the integral in \eqref{e:ineqforenergy} by
\begin{equation} \label{e:ineqforenergy2}
Ck^{n/2}\int_{\R^n\backslash B_d(x(s),r)} e^{-2\alpha' kd(x',x'(s))^2}d\ell_n(x') 
\end{equation}
Making the change of variables $y=k^{-1/2}(y-x(s))$, we bound above \eqref{e:ineqforenergy2} by
\begin{equation*}
C\int_{\R^n\backslash B_d(0,rk^{1/2})} e^{-2\alpha' \|y\|^2}d\ell_n(y)
\end{equation*}
with $\|\cdot\|$ the Euclidean norm. This last expression is bounded above by
\begin{equation*} 
Ce^{-\alpha'r^2k}\int_{\R^n} e^{-\alpha'\|y\|^2}d\ell_n(y)
\end{equation*}
which implies \eqref{e:decreaseenergyforv}.

\paragraph{Extension of the result to any manifold $M$.} In the case of a general manifold $M$, not necessarily included in $\R^n$, we use charts together with the above construction. We cover $M$ by a set of charts $(U_\alpha,\varphi_\alpha)$, where $(U_\alpha)$ is a family of open sets of $M$ covering $M$ and $\varphi_\alpha:U_\alpha\rightarrow \R^n$ is an homeomorphism $U_\alpha$ onto an open subset of $\R^n$. Take a solution $(x(t),\xi(t))_{t\in [0,T]}$ of \eqref{e:bicarac}. It visits a finite number of charts in the order $U_{\alpha_1},U_{\alpha_2},\ldots$, and we choose the charts and $a_0$ so that $v_k(t,\cdot)$ is supported in a unique chart at each time $t$. The above construction shows how to construct $a_0$ and $\psi$ as long as $x(t)$ remains in the same chart. For any $l\geq 1$, we choose $t_l$ so that $x(t_l)\in U_{\alpha_l}\cap U_{\alpha_{l+1}}$ and $a_0(t_l,\cdot)$ is supported in $U_{\alpha_l}\cap U_{\alpha_{l+1}}$. Since there is a (local) solution $v_k$ for any choice of initial $a_0(t_l,x(t_l))$ and $\text{Im}\left(\frac{\partial^2\psi}{\partial x_i\partial x_j}\right)(t_l,x(t_l))$ in Proposition \ref{p:approxgb}, we see that $v_k$ may be continued from the chart $U_{\alpha_l}$ to the chart $U_{\alpha_{l+1}}$. This continuation is smooth since the two solutions coincide as long as $a_0(t,\cdot)$ is supported in $U_{\alpha_l}\cap U_{\alpha_{l+1}}$. Patching all solutions on the time intervals $[t_l,t_{l+1}]$ together, it yields a global in time solution $v_k$, as desired.

\section{Proof of \eqref{e:nullitynu}} \label{a:proofnullitynu}
Because of the second convergence in \eqref{e:twoconvweakobs} and the non-negativity of $a$, it amounts to proving that
\begin{equation*}
(X_1\Op(a)u^k,X_1u^k)_{L^2((0,T)\times M_H)}+(X_2\Op(a)u^k,X_2u^k)_{L^2((0,T)\times M_H)} \rightarrow 0.
\end{equation*}
Now, we notice that for any $B\in\Psi^0_\phg((0,T)\times M_H)$, there holds
\begin{equation} \label{e:Bconv0}
(Bu^k,X_1u^k)_{L^2((0,T)\times M_H)} \underset{k\rightarrow+\infty}{\longrightarrow} 0 \qquad \text{and} \qquad (Bu^k,\partial_tu^k)_{L^2((0,T)\times M_H)} \underset{k\rightarrow+\infty}{\longrightarrow} 0
\end{equation}
since $u^k\rightarrow 0$ strongly in $L^2((0,T)\times M_H)$ and both $X_1u^k$ and $\partial_tu^k$ are bounded in $L^2((0,T)\times M_H)$.
We apply this to $B=[X_1,\Op(a)]$, and then, also using \eqref{e:Bconv0}, we see that we can replace $\Op(a)$ by its Friedrichs quantization $\Op^F(a)$, which is positive (see \cite[Chapter VII]{taylorpseudodifferential}). In other words, we are reduced to prove 
\begin{equation} \label{e:toprovefriedrichs}
(\Op^F(a)X_1u^k,X_1u^k)_{L^2((0,T)\times M_H)}+(\Op^F(a)X_2u^k,X_2u^k)_{L^2((0,T)\times M_H)} \underset{k\rightarrow+\infty}{\longrightarrow} 0.
\end{equation}
Let $\delta>0$ and $\widetilde{a}\in S^0_{\phg}((-\delta,T+\delta)\times M_H)$, $0\leq\widetilde{a}\leq \sup(a)$ and such that $\widetilde{a}(t,\cdot)=a(\cdot)$ for $0\leq t\leq T$. Making repeated use of \eqref{e:Bconv0} and of integrations by parts (since $\widetilde{a}$ is compactly supported in time), we have
\begin{align}
\sum_{j=1}^2(\Op^F(\widetilde{a})X_j u^k, X_j u^k)_{L^2((0,T)\times M_H)}&=\sum_{j=1}^2(X_j \Op^F(\widetilde{a}) u^k, X_ju^k)_{L^2((0,T)\times M_H)}+o(1) \nonumber \\
&=-(\Op^F(\widetilde{a}) u^k,\Delta u^k)_{L^2((0,T)\times M_H)}+\o(1) \nonumber\\
&=-(\Op^F(\widetilde{a}) u^k, \partial_{t}^2u^k)_{L^2((0,T)\times M_H)}+\o(1) \nonumber \\
&=(\partial_t \Op^F(\widetilde{a}) u^k,\partial_{t} u^k)_{L^2((0,T)\times M_H)}+\o(1) \nonumber \\
&= (\Op^F(\widetilde{a})\partial_{t} u^k, \partial_tu^k)_{L^2((0,T)\times M_H)}+\o(1). \nonumber 
\end{align}
Finally we note that since $\Op^F$ is a positive quantization, we have
\begin{align}
\sum_{j=1}^2(\Op^F(a)X_j u^k, X_j u^k)_{L^2((0,T)\times M_H)}&\leq \sum_{j=1}^2(\Op^F(\widetilde{a})X_j u^k, X_j u^k)_{L^2((0,T)\times M_H)} \nonumber \\
&=(\Op^F(\widetilde{a})\partial_{t} u^k, \partial_tu^k)_{L^2((0,T)\times M_H)}+\o(1) \nonumber \\
&\leq C\delta+(\Op^F(a)\partial_{t} u^k, \partial_tu^k)_{L^2((0,T)\times M_H)}+\o(1) \nonumber  \\
&\leq C\delta+o(1) \nonumber
\end{align}
where $C$ does not depend on $\delta$. Making $\delta\rightarrow0$, it concludes the proof of \eqref{e:toprovefriedrichs}, and consequently \eqref{e:nullitynu} holds.


\begin{thebibliography}{CdVHT18}

\bibitem[ABB19]{agrachev2019}
Andrei Agrachev, Davide Barilari and Ugo Boscain.
\newblock {\em A comprehensive introduction to sub-Riemannian geometry.}
\newblock Cambridge University Press, 2019.

\bibitem[AG01]{agrachev2001}
Andrei Agrachev and Jean-Paul Gauthier. 
\newblock On the subanalyticity of Carnot-Caratheodory distances.
\newblock {\em Annales de l'IHP Analyse non linéaire}, 18(3):359--382, 2001.

\bibitem[BC17]{beauchard2017heat}
Karine Beauchard and Piermarco Cannarsa.
\newblock Heat equation on the {H}eisenberg group: Observability and
  applications.
\newblock {\em Journal of Differential Equations}, 262(8):4475--4521, 2017.

\bibitem[BCG14]{BCG14}
Karine Beauchard, Piermarco Cannarsa and Roberto Guglielmi. 
\newblock Null controllability of Grushin-type operators in dimension two. 
\newblock {\em Journal of the European Mathematical Society}, 16(1):67--101, 2014.

\bibitem[BDE20]{beauchard20}
Karine Beauchard, J\'er\'emi Dard\'e and Sylvain Ervedoza.
\newblock Minimal time issues for the observability of Grushin-type equations. 
\newblock {\em Annales de l'Institut Fourier}, 70(1):247--312, 2020.

\bibitem[Bel96]{bellaiche1996tangent}
Andr{\'e} Bella{\"\i}che.
\newblock The tangent space in sub-{R}iemannian geometry.
\newblock In {\em Sub-Riemannian geometry}, pages 1--78. Springer, 1996.

\bibitem[BG02]{burq2002controle}
Nicolas Burq and Patrick G{\'e}rard.
\newblock {\em Contr{\^o}le optimal des equations aux deriv{\'e}es partielles}.
\newblock Ecole polytechnique, D{\'e}partement de math{\'e}matiques, 2002.

\bibitem[BLR92]{bardos1992sharp}
Claude Bardos, Gilles Lebeau, and Jeffrey Rauch.
\newblock Sharp sufficient conditions for the observation, control, and
  stabilization of waves from the boundary.
\newblock {\em SIAM {J}ournal on {C}ontrol and {O}ptimization},
  30(5):1024--1065, 1992.

\bibitem[BS19]{burq2019time}
Nicolas Burq and Chenmin Sun.
\newblock Time optimal observability for {G}rushin {S}chr\"odinger equation.
\newblock To appear in {\em Analysis \& PDE}.

\bibitem[CdVHT18]{de2018spectral}
Yves Colin~de Verdi{\`e}re, Luc Hillairet, and Emmanuel Tr{\'e}lat.
\newblock Spectral asymptotics for sub-{R}iemannian {L}aplacians, {I}: Quantum
  ergodicity and quantum limits in the 3-dimensional contact case.
\newblock {\em Duke Mathematical Journal}, 167(1):109--174, 2018.

\bibitem[DK20]{duprez2018control}
Michel Duprez and Armand Koenig.
\newblock Control of the {G}rushin equation: non-rectangular control region and
  minimal time.
\newblock {\em ESAIM: Control, Optimisation and Calculus of Variations}, volume 26, 2020.

\bibitem[EN99]{engel1999one}
Klaus-Jochen Engel and Rainer Nagel.
\newblock {\em One-parameter semigroups for linear evolution equations}, volume
  194.
\newblock Springer Science \& Business Media, 1999.

\bibitem[G{\'e}r91]{gerard1991mesures}
Patrick G{\'e}rard.
\newblock Mesures semi-classiques et ondes de {B}loch.
\newblock {\em S{\'e}minaire {\'E}quations aux d{\'e}riv{\'e}es partielles
  (Polytechnique)}, pages 1--19, 1991.

\bibitem[GR15]{garetto2015wave}
Claudia Garetto and Michael Ruzhansky.
\newblock Wave equation for sums of squares on compact lie groups.
\newblock {\em Journal of Differential Equations}, 258(12):4324--4347, 2015.

\bibitem[H{\"o}r67]{hormander1967hypoelliptic}
Lars H{\"o}rmander.
\newblock Hypoelliptic second order differential equations.
\newblock {\em Acta Mathematica}, 119(1):147--171, 1967.

\bibitem[H{\"o}r07]{hormander2007analysis}
Lars H{\"o}rmander.
\newblock {\em The analysis of linear partial differential operators III:
  Pseudo-differential operators}.
\newblock Springer Science \& Business Media, 2007.

\bibitem[Ivr19]{ivrii2019microlocal}
Victor Ivrii.
\newblock {\em Microlocal Analysis, Sharp Spectral Asymptotics and Applications
  I: Semiclassical Microlocal Analysis and Local and Microlocal Semiclassical
  Asymptotics}.
\newblock Springer International Publishing, 2019.

\bibitem[Jea14]{jean2014control}
Fr{\'e}d{\'e}ric Jean.
\newblock {\em Control of nonholonomic systems: from sub-Riemannian geometry to
  motion planning}.
\newblock Springer, 2014.

\bibitem[Koe17]{koenig2017non}
Armand Koenig.
\newblock Non-null-controllability of the {G}rushin operator in 2{D}.
\newblock {\em Comptes Rendus Mathematique}, 355(12):1215--1235, 2017.

\bibitem[Las82]{lascar1982}
Bernard Lascar.
\newblock Propagation des singularit\'es pour des \'equations hyperboliques \`a caract\'eristique de multiplicit\'e au plus double et singularit\'es masloviennes. \newblock {\em American Journal of Mathematics}, 104(2):227--285, 1982.

\bibitem[LS20]{Letrouit20}
Cyril Letrouit and Chemin Sun. 
\newblock Observability of Baouendi-Grushin-Type Equations Through Resolvent Estimates. 
\newblock To appear in {\em Journal de l'Institut Math\'ematique de Jussieu}.

\bibitem[Lio88]{lions1988controlabilite}
Jacques-Louis Lions.
\newblock Contr{\^o}labilit{\'e} exacte, stabilisation et perturbations de
  systemes distribu{\'e}s. {T}ome 1. {C}ontr{\^o}labilit{\'e} exacte.
\newblock {\em Rech. Math. Appl}, 8, 1988.

\bibitem[LL20]{laurent2017tunneling}
Camille Laurent and Matthieu L{\'e}autaud.
\newblock Tunneling estimates and approximate controllability for hypoelliptic
  equations.
\newblock {\em To appear in Mem. Am. Math. Soc.}

\bibitem[LRLTT17]{le2017geometric}
J{\'e}r{\^o}me Le~Rousseau, Gilles Lebeau, Peppino Terpolilli, and Emmanuel
  Tr{\'e}lat.
\newblock Geometric control condition for the wave equation with a
  time-dependent observation domain.
\newblock {\em Analysis \& PDE}, 10(4):983--1015, 2017.

\bibitem[Mil12]{miller2012}
Luc Miller.
\newblock Resolvent conditions for the control of unitary groups and their approximations. 
\newblock {\em Journal of Spectral Theory}, 2(1):1--55, 2012.

\bibitem[Mit85]{mitchell85}
John Mitchell.
\newblock On Carnot-Caratheodory metrics. 
\newblock {\em Journal of Differential Geometry}, 21(1):35--45, 1985.


\bibitem[Mon02]{montgomery2002tour}
Richard Montgomery.
\newblock {\em A tour of subriemannian geometries, their geodesics and
  applications}.
\newblock Number~91. American Mathematical Soc., 2002.

\bibitem[MZ02]{macia2002lack}
Fabricio Macia and Enrique Zuazua.
\newblock On the lack of observability for wave equations: a {G}aussian beam
  approach.
\newblock {\em Asymptotic Analysis}, 32(1):1--26, 2002.

\bibitem[Ral82]{ralston1982gaussian}
James Ralston.
\newblock Gaussian beams and the propagation of singularities.
\newblock {\em Studies in partial differential equations}, 23(206):C248, 1982.

\bibitem[RS76]{rothschild1976hypoelliptic}
Linda~Preiss Rothschild and Elias~M. Stein.
\newblock Hypoelliptic differential operators and nilpotent groups.
\newblock {\em Acta Mathematica}, 137(1):247--320, 1976.

\bibitem[Str86]{strichartz1986sub}
Robert~S. Strichartz.
\newblock Sub-{R}iemannian geometry.
\newblock {\em Journal of Differential Geometry}, 24(2):221--263, 1986.

\bibitem[Tar90]{tartar1990h}
Luc Tartar.
\newblock H-measures, a new approach for studying homogenisation, oscillations
  and concentration effects in partial differential equations.
\newblock {\em Proceedings of the Royal Society of Edinburgh Section A:
  Mathematics}, 115(3-4):193--230, 1990.

\bibitem[Tay74]{taylorpseudodifferential}
Michael Taylor.
\newblock Pseudodifferential operators.
\newblock {\em Lecture notes in Mathematics}, 1974.

\bibitem[TW09]{tucsnak2009observation}
Marius Tucsnak and George Weiss.
\newblock {\em Observation and control for operator semigroups}.
\newblock Springer Science \& Business Media, 2009.

\end{thebibliography}
\end{document}